\documentclass{amsart}
\usepackage[ansinew]{inputenc}
\usepackage[english]{babel}
\usepackage{amsmath,latexsym,amssymb,verbatim}
\usepackage{pdfsync}

\usepackage[all]{xy}

\usepackage{hyperref}

\newtheorem{theorem}{Theorem}[section]
\newtheorem{lemma}[theorem]{Lemma}
\newtheorem{proposition}[theorem]{Proposition}
\newtheorem{corollary}[theorem]{Corollary}
\newtheorem{note}[theorem]{Note}
\newtheorem{Formula of adjoint functors}[theorem]{Formula of adjoint functors}

\newtheorem{example}[theorem]{Example}
\newtheorem{examples}[theorem]{Examples}
\newtheorem{definition}[theorem]{Definition}
\newtheorem{notation}[theorem]{Notation}

\DeclareMathOperator{\limi}{{lim}}
\newcommand{\ilim}[1]{\,\underset{#1}{\underset{\to}{\limi}}\,}
\newcommand{\plim}[1]{\,\underset{#1}{\underset{\leftarrow}{\limi}}\,}

\DeclareMathOperator{\Hom}{{Hom}}
\DeclareMathOperator{\Spec}{{Spec}}
\DeclareMathOperator{\Coker}{{Coker}}
\DeclareMathOperator{\Ker}{{Ker}}
\DeclareMathOperator{\Ima}{{Im}}

\NoCompileMatrices

\input arrow.tex

\begin{document}

\title{Reflexive functors of modules in Commutative Algebra}

\author{José Navarro}
\address[José Navarro]{Departamento de Matemáticas, Universidad de Extremadura,
Avenida de Elvas s/n, 06071 Badajoz, Spain}
\email{navarrogarmendia@unex.es}

\author{Carlos Sancho}
\address[Carlos Sancho]{Departamento de Matemáticas, Universidad de Salamanca,
Plaza de la Merced 1-4, 37008 Salamanca, Spain}
\email{mplu@usal.es}

\author{Pedro Sancho}
\address[Pedro Sancho]{Departamento de Matemáticas, Universidad de Extremadura,
Avenida de Elvas s/n, 06071 Badajoz, Spain}
\email{sancho@unex.es}
\thanks{Corresponding Author: Pedro Sancho.}

\date{April 12, 2012}

\begin{abstract}
Reflexive functors of modules naturally appear in Algebraic Geometry, mainly in the theory of linear representations of group schemes, and in ``duality theories''.
In this paper we study and determine reflexive functors of modules and we give many properties of reflexive functors of modules, of algebras and of bialgebras.
\end{abstract}

\maketitle

\section{Introduction}

Let $X=\Spec A$ be an affine scheme over a field $K$. We can regard $X$ as a covariant functor of
sets over the category of commutative $K$-algebras through its functor of points $X^\cdot$, defined by $X^\cdot(S):=\Hom_{K-alg}(A,S)$, for all commutative $K$-algebras $S$. If $X=\Spec K[x_1,\ldots,x_n]/(p_1,\ldots,p_m)$ then
$$X^\cdot(S):=\{s\in S^n\colon p_1(s)=\cdots=p_m(s)=0\}$$
By the Yoneda Lemma, $\Hom_{K-sch}(X,Y)=\Hom_{funct.}(X^\cdot,Y^\cdot)$, and
it is well known that $X$ is an affine group $K$-scheme if and only if $X^\cdot$ is a functor of groups.

We can regard $K$ as functor of rings $\mathcal K$,  by defining   $\mathcal K(S):=S$, for all commutative $K$-algebras $S$.
Let $V$ be a $K$-vector space. We can regard $V$ as a covariant functor of $\mathcal K$-modules, $\mathcal V$, by defining $\mathcal V(S):=V\otimes_K S$. We will say that $\mathcal V$ is the $\mathcal K$-quasi-coherent module associated with $V$.  If
$V=\oplus_IK$ then $\mathcal V(S)=\oplus_I S$. It holds that the category of $K$-vector spaces, ${\mathcal C}_{\text{K-vect}}$, is equivalent to the category of quasi-coherent $\mathcal K$-modules, ${\mathcal C}_{\text{qs-coh $\mathcal K$-mod}}$: the functors ${\mathcal C}_{\text{K-vect}} \rightsquigarrow {\mathcal C}_{\text{qs-coh $\mathcal K$-mod}}$, $V\rightsquigarrow \mathcal V$ and ${\mathcal C}_{\text{qs-coh $\mathcal K$-mod}}\rightsquigarrow{\mathcal C}_{\text{K-vect}},$ $\mathcal V\rightsquigarrow \mathcal V(K)$ give the equivalence.

It is well known that the theory of linear representations of a group scheme $G=\Spec A$ can be developed, via their associated functors, as a theory  of an abstract group and its linear representations. That is,   the category of linear representations of a group scheme $G$ is equivalent to the category of quasi-coherent $G^\cdot$-modules.

Given a functor of $\mathcal K$-modules, $\mathbb M$ (that is, a covariant functor from the category of commutative $K$-algebras to the category of abelian groups, with a structure of $\mathcal K$-module), we denote $\mathbb M^*:=\mathbb Hom_{\mathcal K}(\mathbb M,\mathcal K)$. We say that $\mathbb M$ is a reflexive functor of modules if $\mathbb M=\mathbb M^{**}$.

 Reflexive functors of modules naturally appear in Algebraic Geometry, mainly in the theory of linear representations of group schemes, and in ``duality theories'':
 Quasi-coherent modules are reflexive (surprisingly even when $K$ is a commutative  ring and $V$ is a non finitely generated $K$-module, see \cite{Amel}).
Let $\mathbb X$ be a functor of sets and $\mathbb A_{\mathbb X}:=\mathbb Hom_{funct.}(\mathbb X,\mathcal K)$. We say that $\mathbb X$ is an affine functor if  $\mathbb A_{\mathbb X}$ is reflexive and $\mathbb X=\Spec \mathbb A_{\mathbb X}:=\mathbb Hom_{\mathcal K-alg}(\mathbb A_{\mathbb X},\mathcal K)$, see \cite{navarro} for details (we warn the reader that  in the literature affine functors are sometimes defined to be functors of points of affine schemes). In \cite{navarro}, we prove that (functors of points of) affine  schemes, formal schemes and the completion of an affine scheme along a closed set
are affine functors. Let $\mathbb G$ be an affine functor of monoids. $\mathbb A_{\mathbb G}^*$ is a functor of algebras and the category of $\mathbb G$-modules is equivalent
to the category of $\mathbb A_{\mathbb G}^*$-modules. Applications of these results include Cartier duality, neutral Tannakian duality for affine group schemes and the equivalence between formal groups and Lie algebras in characteristic zero (see \cite{navarro}).
In order to prove these results it is necessary  to study and to determine reflexive functors of modules, algebras and bialgebras.

Some natural questions emerge:
Is the family of reflexive functors a monster family?  Is this family closed under tensor products? Is this family closed under homomorphisms?

In this paper we prove: \begin{enumerate}

\item Each reflexive functor of $\mathcal K$-modules is a functor of $\mathcal K$-submodules of a funtor of $\mathcal K$-modules $\prod_I\mathcal K$, for a certain set $I$ (see \ref{2125}).

\item A functor of $\mathcal K$-modules is reflexive if and only if it is the inverse limit of its quasi-coherent quotients (see \ref{RPQ}).

\item    If $I$ is a totally ordered set and $\{f_{ij}\colon \mathcal V_i\to \mathcal V_j\}_{i\geq j\in I}$ is an inverse system of quasi-coherent $\mathcal K$-modules, then $\plim{i\in I} \mathcal V_i$ is a reflexive functor of $\mathcal K$-modules.

\noindent We do not know if arbitrary inverse limits of quasi-coherent modules are reflexive, that is, if proquasi-coherent modules are reflexive.

\item If $\mathbb M$ and $\mathbb M'$ are reflexive functors of $\mathcal K$-modules, then
$\Hom_{\mathcal K}(\mathbb M,\mathbb M')\subseteq \Hom_{K}(\mathbb M(K),\mathbb M'(K))$ ({see \ref{QHom}}). If $\mathbb A$ is a reflexive functor and a functor of $\mathcal K$-algebras and
    $\mathbb M,\mathbb M'$ are reflexive functors of $\mathbb A$-modules, then a morphism
    of $\mathcal K$-modules $\mathbb M\to\mathbb M'$ is a morphism of $\mathbb A$-modules if and only if $\mathbb M(K)\to\mathbb M'(K)$ is a morphism of $\mathbb A(K)$-modules. Let $V$ be a vector space. If $\mathcal V$ is an $\mathbb A$-module, then the set of all quasi-coherent $\mathbb A$-submodules of $\mathcal V$ is equal to the set of all $\mathbb A(K)$-submodules of $V$
    (see \ref{P10} and \ref{P9}).
\end{enumerate}

 Now assume $K=R$ is a commutative ring. In section \ref{F}, we define a wide family $\mathfrak F$ of reflexive functors of $\mathcal R$-modules satisfying:

 \begin{enumerate}

 \item Each $\mathbb M\in \mathfrak F$ is a functor of $\mathcal R$-submodules of a funtor of $\mathcal R$-modules $\prod_I\mathcal R$, for a certain set $I$.

 \item If $M$ and $N$ are free $R$-modules, then $\mathcal M,\mathcal M^*,\mathbb Hom_{\mathcal R}(\mathcal M,\mathcal N)\in\mathfrak F$.

\item  Every functor of $\mathcal R$-modules $\mathbb M\in\mathfrak F$  is proquasi-coherent.

 \item  If $\mathbb M,\mathbb M'\in\mathfrak F$, then $\mathbb Hom_{\mathcal R}(\mathbb M,\mathbb M')\in\mathfrak F$ and
$(\mathbb M\otimes_{\mathcal R}\mathbb M')^{**}\in\mathfrak F$, which satisfies
     $$
     \Hom_{\mathcal R}((\mathbb M\otimes_{\mathcal R}\mathbb M')^{**},\mathbb {M''})=\Hom_{\mathcal R}(\mathbb M\otimes_{\mathcal R}\mathbb M',\mathbb {M''})$$
for every reflexive  functor of $\mathcal R$-modules, $\mathbb {M''}$.

\item If $\mathbb A,\mathbb B\in\mathfrak F$ are functors of proquasi-coherent algebras, then $(\mathbb A^*\otimes_{\mathcal R} \mathbb B^*)^*\in\mathfrak F$ and it is a functor of proquasi-coherent algebras, which satisfies
     $$\Hom_{\mathcal R-alg}((\mathbb A^*\otimes_{\mathcal R}\mathbb B^*)^*,\mathbb C)=
     \Hom_{\mathcal R-alg}(\mathbb A\otimes_{\mathcal R}\mathbb B,\mathbb C)$$
for every functor of proquasi-coherent algebras, $\mathbb C$.

\item The functor ${\mathcal C}_{\mathfrak F-bialg} \rightsquigarrow {\mathcal C}_{\mathfrak F-bialg}$, $ \mathbb B\rightsquigarrow \mathbb B^*$ is a categorical anti-equivalence, where
${\mathcal C}_{\mathfrak F-bialg}$ is the category of  functors of proquasi-coherent bialgebras (that is, $\mathbb B\in {\mathcal C}_{\mathfrak F-bialg}$ if $\mathbb B\in \mathfrak F$, $\mathbb B$ and $\mathbb B^*$ are functors of proquasi-coherent algebras and the dual morphisms of the multiplication morphism and unit morphism of $\mathbb B^*$ are morphisms of functors of algebras).

Let $A$ be a free $R$-module, then $\mathcal A\in \mathfrak F$. $A$ is an $R$-bialgebra if and only if $\mathcal A$ is a functor of proquasi-coherent bialgebras (see Proposition \ref{5.24}).
In the literature, there have been many attempts to obtain a well-behaved duality for non finite dimensional bialgebras (see \cite{Timmerman} and references
therein). One of them, for example, states that the functor that associates with  each bialgebra $A$ over a field $K$ the so-called dual bialgebra $A^\circ$ is auto-adjoint ($A^\circ:=\ilim{I\in J} (A/I)^*$, where $J$ is the set of bilateral ideals $I\subset A$ such that $\dim_K A/I<\infty$, see \cite{E}).
Another one associates with each bialgebra
$A$ over a pseudocompact ring $R$ the bialgebra $A^*$ endowed with a certain topology (see \cite[Exposé VII$_B$ 2.2.1]{demazure}).

\item If $\mathbb M,\mathbb M'\in\mathfrak F$, then $\Hom_{\mathcal R}(\mathbb M,\mathbb M')\subseteq \Hom_{R}(\mathbb M(R),\mathbb M'(R))$.
    If $\mathbb A\in \mathfrak F$ is a functor of $\mathcal R$-algebras and
    $\mathbb M,\mathbb M'\in\mathfrak F$ are functors of $\mathbb A$-modules, then a morphism
    of $\mathcal R$-modules $\mathbb M\to\mathbb M'$ is a morphism of $\mathbb A$-modules if and only if $\mathbb M(R)\to\mathbb M'(R)$ is a morphism of $\mathbb A(R)$-modules. Let $M$ be an $R$-module. If $\mathcal M$ is an $\mathbb A$-module, then the set of all quasi-coherent $\mathbb A$-submodules of $\mathcal M$ is equal to the set of all $\mathbb A(R)$-submodules of $M$.
    \end{enumerate}

This paper completes \cite{Amel} and it is essentially self contained.

\section{Preliminaries}

Let $R$ be a commutative ring (associative with a unit). All functors considered in this paper are covariant functors over the category of commutative $R$-algebras (always assumed to be associative with a unit). A functor $\mathbb X$ is said to be a functor of sets (resp. monoids, etc.) if $\mathbb X$ is a functor from the category of  commutative $R$-algebras to the category of sets (resp. monoids, etc.).

\begin{notation} For simplicity, given a functor of sets $\mathbb X$,
we sometimes use $x \in \mathbb X$  to denote $x \in \mathbb X(S)$. Given $x \in \mathbb X(S)$ and a morphism of commutative $R$-algebras $S \to S'$, we still denote by $x$ its image by the morphism $\mathbb X(S) \to \mathbb X(S')$.\end{notation}

Let $\mathcal R$ be the functor of rings defined by ${\mathcal R}(S):=S$, for all commutative $R$-algebras $S$.
A functor of sets $\mathbb M$ is said to be a functor of $\mathcal R$-modules if we have morphisms of functors of sets, $\mathbb M\times \mathbb M\to \mathbb M$ and ${\mathcal R}\times \mathbb M\to \mathbb M$, so that
$\mathbb M(S)$ is an $S$-module, for every commutative $R$-algebra $S$.
A functor of rings (associative with a unit), $\mathbb A$,
is said to be a functor of $\mathcal R$-algebras if we have a morphism of functors of rings
$\mathcal R\to \mathbb A$
(and $\mathcal R(S)=S$ commutes with all the elements of $\mathbb A(S)$, for every commutative $R$-algebra $S$).

Given a commutative $R$-algebra $S$, we denote by $\mathbb M_{|S}$ the functor $\mathbb M$ restricted to the category of commutative $S$-algebras.

 Let $\mathbb M$ and $\mathbb M'$ be functors of $\mathcal R$-modules.
 A morphism of functors of $\mathcal R$-modules $f\colon \mathbb M\to \mathbb M'$
 is a morphism of functors such that the defined morphisms $f_S\colon \mathbb M(S)\to
 \mathbb M'(S)$ are morphisms of $S$-modules, for all commutative $R$-algebras $S$.
 We will denote by $\Hom_{\mathcal R}(\mathbb M,\mathbb M')$ the  set (see footnote) of all morphisms of $\mathcal R$-modules from $\mathbb M$ to $\mathbb M'$.

\begin{example} \label{ex1} $\Hom_{\mathcal R}(\prod^{\mathbb N}\mathcal R, \mathcal R)=\oplus^{\mathbb N} R$: Let $w\colon \prod^{\mathbb N}\mathcal R\to \mathcal R$ be an $\mathcal R$-linear morphism. Let us consider the polynomial ring  $S:=R[x_1,\ldots,x_n,\ldots]$. Since
 $w_S((x_1,\ldots,x_n,\ldots))\in S$,
for some $m\in\mathbb N$, $w_S((x_1,\ldots,x_n,\ldots))\in R[x_1,\ldots,x_m]$. Let $$\phi(x_1,\ldots,x_m):=w_S((x_1,\ldots,x_n,\ldots)).$$ Let $S'$ be a commutative $R$-algebra and
$s'_1,\ldots,s'_n,\ldots \in S'$. Let us consider the  morphism of $R$-algebras
$S\to S'$, $x_i\mapsto s'_i$ (for all $i$). By functoriality, $$w_{S'}((s'_1,\ldots,s'_n,\cdots))=\phi(s'_1,\ldots,s'_m).$$
Hence, $w$ is determined by $\phi(x_1,\ldots,x_m)$.
Finally,
$\phi(x_1,\ldots, x_m)=\lambda_1 x_1+\cdots+\lambda_m x_m$, for certain
$\lambda_1,\ldots,\lambda_m\in R$, because
$$\phi(x\cdot x_1,\ldots,x\cdot x_m)=x\cdot \phi(x_1,\ldots,x_m)\in R[x,x_1,\ldots,x_m]=:T$$
since $w_T$ is $T$-linear.

\end{example}

We will denote by ${\mathbb Hom}_{\mathcal R}(\mathbb M,\mathbb M')$\footnote{In this paper, we will only  consider well defined functors ${\mathbb Hom}_{\mathcal R}(\mathbb M,\mathbb M')$, that is to say, functors such that $\Hom_{\mathcal S}(\mathbb M_{|S},\mathbb {M'}_{|S})$ is a set, for all $S$.} the functor of $\mathcal R$-modules $${\mathbb Hom}_{\mathcal R}(\mathbb M,\mathbb M')(S):={\rm Hom}_{\mathcal S}(\mathbb M_{|S}, \mathbb M'_{|S})$$  Obviously,
$$(\mathbb Hom_{\mathcal R}(\mathbb M,\mathbb M'))_{|S}=
\mathbb Hom_{\mathcal S}(\mathbb M_{|S},\mathbb M'_{|S})$$

\begin{notation} We denote $\mathbb M^*=\mathbb Hom_{\mathcal R}(\mathbb M,\mathcal R)$.\end{notation}

\begin{example} $(\oplus^{\mathbb N} \mathcal R)^{**}=(\prod^{\mathbb N} \mathcal R)^*\underset{\text{\ref{ex1}}}=\oplus^{\mathbb N}\mathcal R$.\end{example}

\begin{notation} Tensor products, direct limits, inverse limits, etc., of functors of $\mathcal R$-modules and  kernels, cokernels, images, etc.,  of morphisms of functors of $\mathcal R$-modules are regarded in the category of functors of $\mathcal R$-modules.\end{notation}

It holds that
$$\aligned & (\mathbb M\otimes_{\mathcal R} \mathbb M')(S)=\mathbb M(S)\otimes_{S} \mathbb M'(S),\,
(\Ker f)(S)=\Ker f_S,\, (\Coker f)(S)=\Coker f_S,\\
& (\Ima f)(S)=\Ima f_S,\, (\ilim{i\in I} \mathbb M_i)(S)=\ilim{i\in I} (\mathbb M_i(S)),\,
(\plim{i\in I} \mathbb M_i)(S)=\plim{i\in I} (\mathbb M_i(S))\endaligned $$

\begin{definition} Given an $R$-module $M$ (resp. $N$, etc.), ${\mathcal M}$  (resp. $\mathcal N$, etc.) will denote  the functor of $\mathcal R$-modules  defined by ${\mathcal M}(S) := M \otimes_R S$ (resp. $\mathcal N(S):=N\otimes_R S$, etc.). $\mathcal M$  will be called  quasi-coherent $\mathcal R$-module (associated with $M$).
\end{definition}

\begin{proposition} \cite[1.3]{Amel}\label{tercer}
For every functor of ${\mathcal R}$-modules $\mathbb M$ and every $R$-module $M$, it
holds that
$${\rm Hom}_{\mathcal R} ({\mathcal M}, \mathbb M) = {\rm Hom}_R (M, \mathbb M(R))$$
\end{proposition}

\begin{proof}
Given an ${\mathcal R}$-linear morphism $f: {\mathcal M} \to \mathbb M$, we have for every
$R$-algebra $S$ a morphism of $S$-modules $f_S: M \otimes_R S \to
\mathbb M(S)$ and a commutative diagram $$\begin{matrix}
M \otimes_R S  & \stackrel{f_S}{\longrightarrow} & \mathbb M(S) \\
\uparrow & & \uparrow \\ M & \stackrel{f_R}{\longrightarrow} &
\mathbb M(R)
\end{matrix}$$ Hence, the morphism of $S$-modules $f_S$ is determined by $f_R$.
\end{proof}

The functors $M \rightsquigarrow {\mathcal M}$, ${\mathcal M} \rightsquigarrow {\mathcal M}(R)=M$ establish an equivalence between the category of $R$-modules and the category of quasi-coherent $\mathcal R$-modules (\cite[1.12]{Amel}). In particular, ${\rm Hom}_{\mathcal R} ({\mathcal M},{\mathcal M'}) = {\rm Hom}_R (M,M')$.
 For any pair of $R$-modules $M$ and $N$, the quasi-coherent module associated with $M\otimes_R N$ is $\mathcal M\otimes_{\mathcal R}\mathcal N$. ${\mathcal M}_{\mid S}$ is the quasi-coherent $\mathcal S$-module associated with $M \otimes_R
S$

\begin{definition} The functor ${\mathcal M}^* = {\mathbb Hom}_{\mathcal R} ({\mathcal M}, {\mathcal R})$ is  called an $\mathcal R$-module scheme. \end{definition}

${\mathcal M}^*(S)={\rm Hom}_S(M\otimes_RS,S)={\rm Hom}_R(M,S)$ and it is easy to check that $(\mathcal M^*)_{|S}$ is an $\mathcal S$-module scheme.

\begin{definition}
Given a commutative $R$-algebra $A$, let $({\rm
Spec}\,A)^\cdot$ be the functor defined by $({\rm Spec}\,A)^\cdot (S) := {\rm
Hom}_{\rm R-alg} (A, S)$, for each commutative $R$-algebra
$S$. This functor will be called the functor of points of ${\rm
Spec}\,A$.

By Yoneda's lemma (see \cite[Appendix A5.3]{eisenbud}), ${\rm
Hom}_{\rm func}(({\rm Spec}\,A)^\cdot, \mathbb X) = \mathbb X(A)$.
\end{definition}

Given an $R$-module $M$, we will denote by $S^\cdot_R M$ the
symmetric algebra of $M$. Let us recall the next well-known lemma
(see \cite[II, $\S 1$, 2.1]{gabriel} or \cite[Exposé VII$_B$,
1.2.4]{demazure}).

\begin{lemma} \cite[1.6]{Amel} \label{lema}
If $M$ is an $R$-module, then ${\mathcal M^*} = ({\rm
Spec}\,S^\cdot_R M)^\cdot$ as functors of ${\mathcal R}$-modules.
\end{lemma}

\begin{proof}
For every commutative $R$-algebra $S$, it holds that $${\mathcal M^*}(S)= {\rm Hom}_R (M, S) = {\rm Hom}_{\rm R-alg} (S^\cdot_R M,
S) = ({\rm Spec}\,S^\cdot_R M)^\cdot(S)$$
\end{proof}

\begin{proposition}  \cite[1.8]{Amel}\label{prop4} \label{1.8Amel}
Let $M$, $M'$ be $R$-modules. Then $${\mathbb Hom}_{\mathcal R} ({\mathcal M^*}, {\mathcal M'}) = {\mathcal M} \otimes_{\mathcal R} {\mathcal M'}$$
\end{proposition}

\begin{proof} Note first that

$${\rm Hom}_{\rm func} ({\mathcal M^*}, {\mathcal M'}) \overset{\text{\ref{lema}}}=
{\rm Hom}_{\rm func} (({\rm
Spec}\,S^\cdot_R M)^\cdot, {\mathcal M'})={\mathcal M'}(S^\cdot_R M) = S^\cdot_R M \otimes_R M'.$$
Specifically, $f=m_1\cdots m_n\otimes m'\in S^n_RM\otimes M'$ defines the morphism
of functors $f\colon {\mathcal M^*}\to {\mathcal M'}$, $f(w):=w(m_1)\cdots w(m_n)\cdot m'$, for all $w\in\mathcal M^*$.

In
order for $f \in S^\cdot_R M \otimes_R M'$ to define a $\mathcal R$-linear
morphism of functors, it must be $f \in M \otimes_R M'$: Write $f=\sum_n f_n\in \oplus_{n\in\mathbb N} S_R^n M\otimes_R M'$. Let $S$ be a commutative $R$-algebra and let $w\in\mathcal M^*(S)$. Let $S'=S[x]$ and consider the natural injective morphism
$M'\otimes_R S\hookrightarrow M'\otimes_R S'$ as an inclusion of sets. Then,
$$\aligned x\cdot\sum_n  (f_n)_{S}(w)  & =x\cdot f_{S}(w)= x\cdot f_{S'}(w)  =
f_{S'}(x\cdot w)=\sum_n (f_n)_{S'}(x\cdot w)\\ & =\sum_n x^n\cdot (f_n)_{S'}(w)=
\sum_n x^n\cdot (f_n)_{S}(w)\in M'\otimes_R S'\endaligned$$
Hence $(f_n)_S(w)=0$, for all $n\neq 1$, and then $f\in M\otimes M'$.

We have proved that ${\rm
Hom}_{\mathcal R} ({\mathcal M^*}, {\mathcal M}') = M \otimes_R M'$.
For every $R$-algebra $S$,

\begin{equation*} \begin{split}
{\mathbb Hom}_{\mathcal R} ({\mathcal M^*}, {\mathcal M'})(S) & =  {\rm Hom}_S ({\mathcal M^*}_{\mid S}, {\mathcal M'}_{\mid S}) = {\rm Hom}_S (({\mathcal M}
\otimes_{\mathcal R} {\mathcal S})^*, {\mathcal M'} \otimes_{\mathcal R} {\mathcal S}) \\ & = (M
\otimes_R S) \otimes_S (M' \otimes_R S) = ({\mathcal M} \otimes_{\mathcal R} {\mathcal M'})(S)
\end{split}
\end{equation*}
\end{proof}

If $\mathcal M'=\mathcal R$, in the previous proposition, we obtain the following theorem.

\begin{theorem} \cite[1.10]{Amel}\label{reflex}
Let $M$ be an $R$-module. Then $${\mathcal M^{**}} = {\mathcal M}$$
\end{theorem}

 The functors $\mathcal M\rightsquigarrow \mathcal M^*$ and $\mathcal M^{*}\rightsquigarrow \mathcal M^{**}=\mathcal M$ establish an anti-equivalence between the categories of quasi-coherent modules
and module schemes. An $\mathcal R$-module scheme ${\mathcal M}^*$ is a quasi-coherent $\mathcal R$-module if and only if $M$ is a projective finitely generated  $R$-module (see \cite{Amel2}).

Let us recall the Formula of adjoint functors.

\begin{notation} Let $i\colon R\to S$ be a commutative $R$-algebra.
Given a functor of $\mathcal R$-modules, $\mathbb M$, let $i^* \mathbb M$  be the functor of
$\mathcal S$-modules defined by
$(i^* \mathbb M)(S') := \mathbb M(S')$.
Given a functor of $\mathcal S$-modules, $\mathbb M'$, let $i_* \mathbb M'$ be the functor of
$\mathcal R$-modules defined by
$(i_* \mathbb M')(R') := \mathbb M(S \otimes_R R')$.
\end{notation}

\begin{Formula of adjoint functors} \cite[1.12]{Amel}\label{adj}
Let $\mathbb M$ be a functor of ${\mathcal R}$-modules and let $\mathbb M'$ be a functor of  $\mathcal S$-modules. Then, it holds that
$${\rm Hom}_{\mathcal S} (i^*\mathbb M, \mathbb M') = {\rm Hom}_{\mathcal R} (\mathbb M, i_*\mathbb M')$$
\end{Formula of adjoint functors}

\begin{proof}
Given a $w \in {\rm Hom}_S (i^*\mathbb M, \mathbb M' )$, we have morphisms $w_{S
\otimes R'} : \mathbb M(S \otimes R') \to \mathbb M'(S \otimes R')$ for each
commutative $R$-algebra $R'$. By composition with the morphisms $\mathbb M(R') \to \mathbb M(S
\otimes R')$, we have the morphisms $\phi_{R'}: \mathbb M(R') \to \mathbb M'(S
\otimes R') = i_* \mathbb M'(R')$, which in their turn define $\phi \in
{\rm Hom}_{\mathcal R} (\mathbb M, i_*\mathbb M')$.

Given a $\phi \in {\rm Hom}_{\mathcal R} (\mathbb M, i_*\mathbb M')$, we have morphisms
$\phi_{S'}: \mathbb M(S') \to i_*\mathbb M'(S')=\mathbb M'(S \otimes S')$ for each
$S$-algebra $S'$. By composition with the morphisms $\mathbb M'(S \otimes
S') \to \mathbb M'(S')$, we have the morphisms $w_{S'} : \mathbb M(S') \to \mathbb M'(S')$,
which in their turn define $w \in {\rm Hom}_S (i^*\mathbb M, \mathbb M')$.

Now we shall show that the assignments $w \mapsto \phi$ and $\phi \mapsto w$ are
mutually inverse. Given $w \in {\rm Hom}_S (i^*\mathbb M, \mathbb M')$ we have
$\phi \in {\rm Hom}_{\mathcal R} (\mathbb M, i_*\mathbb M')$. Let us prove that the latter
defines $w$ again. We have the following diagram, where $S'$ is a commutative
$S$-algebra and $q$, $p$ the obvious morphisms,

$$\xymatrix{ \mathbb M(S') \ar[r]^-q \ar@{=}[rd] & \mathbb M(S\otimes S') \ar[d] \ar[r]^-{w_{S\otimes S'}} & \mathbb M'(S\otimes
S') \ar[d]^-p \\ & \mathbb M(S') \ar[r]^-{w_{S'}} & \mathbb M'(S')}$$

The composite morphism $p\circ w_{S\otimes S'}\circ q=p \circ
\phi_{S'}$ is that assigned to $\phi$, and coincides with $w_{S'}$
since the whole diagram is commutative.

Given $\phi \in {\rm Hom}_{\mathcal R} (\mathbb M, i_*\mathbb M')$ we have $w \in {\rm Hom}_{\mathcal S}
(i^*\mathbb M, \mathbb M')$. Let us see that the latter defines $\phi$. We have the
following diagram, where $R'$ is a commutative $R$-algebra and $r,j,p$ the obvious morphisms,

$$\xymatrix{ \mathbb M(R') \ar[r]^-r \ar[d]^-{\phi_{R'}} &
\mathbb M(S\otimes R') \ar[r]^-{w_{ S\otimes R'}} \ar[d]^-{\phi_{S\otimes R'}} &
\mathbb M'(S\otimes R')\\ (i_*\mathbb M')(R') \ar[r]^-j & (i_*\mathbb M')(S\otimes R') \ar@{=}[r] &
\mathbb M'(S\otimes S\otimes R') \ar[u]^-{p} }$$

The composite morphism $w_{S\otimes R'} \circ r$ assigned to $w$
agrees with $\phi_{R'}$, since $p \circ j = Id$ and the whole
diagram is commutative.

\end{proof}

\begin{corollary} \label{adj2} Let $\mathbb M$ be a functor of $\mathcal R$-modules. Then
$$\mathbb M^*(S)=\Hom_{\mathcal R}(\mathbb M,\mathcal S)$$
for all commutative $R$-algebras $S$.\end{corollary}

\begin{proof} $\mathbb M^*(S)=\Hom_{\mathcal S}(\mathbb M_{|S},\mathcal S)
\overset{\text{\ref{adj}}}=\Hom_{\mathcal R}(\mathbb M,\mathcal S)$.
\end{proof}

\begin{definition} Let $\mathbb M$ be a functor of $\mathcal R$-modules. We will say that
$\mathbb M^*$ is a dual functor.
We will say that a functor of $\mathcal R$-modules ${\mathbb M}$ is reflexive if ${\mathbb M}={\mathbb M}^{**}$.\end{definition}

\begin{examples}  Quasi-coherent modules and module schemes are reflexive functors of $\mathcal R$-modules.\end{examples}

\begin{proposition} \label{3.2}
Let ${\mathbb M}$ be a functor of $\mathcal R$-modules such that ${\mathbb M}^*$ is a reflexive
functor. The closure of dual functors of $\mathcal R$-modules of ${\mathbb M}$ is
${\mathbb M}^{**}$, that is, it holds the functorial equality $${\rm
Hom}_\mathcal R({\mathbb M},{\mathbb M}')={\rm Hom}_\mathcal R({\mathbb M}^{**},{\mathbb M'})$$ for every dual functor of
$\mathcal R$-modules ${\mathbb M}'$.
\end{proposition}

\begin{proof} Write $\mathbb M'=\mathbb {N}^*$. Then,
${\rm Hom}_{\mathcal R}({\mathbb M},{\mathbb M}')\! =\!{\rm Hom}_{\mathcal R}({\mathbb M}\otimes{\mathbb N},\mathcal R)\!=\!{\rm Hom}_{\mathcal R}({\mathbb N},{\mathbb M}^*)={\rm
Hom}_{\mathcal R}({\mathbb N}\otimes{\mathbb M}^{**},\mathcal R) ={\rm Hom}_{\mathcal R}({\mathbb M}^{**},{\mathbb M}').$
\end{proof}

%
%
%

\begin{proposition}\label{2.4}
Let $\mathbb A$ be a functor of $\mathcal R$-algebras such that $\mathbb A^*$ is a
reflexive functor of $\mathcal R$-modules. The closure of dual functors of
$\mathcal R$-algebras of $\mathbb A$ is $\mathbb A^{**}$, that is, it holds the functorial
equality $${\rm Hom}_{\mathcal R-alg}(\mathbb A,\mathbb B)={\rm Hom}_{\mathcal R-alg}(\mathbb A^{**},\mathbb B)$$
for every dual functor of $\mathcal R$-algebras $\mathbb B$.

As a consequence, the category of dual functors of ${\mathbb A}$-modules is equal to the category of dual functors of ${\mathbb A}^{**}$-modules.
\end{proposition}

\begin{proof}
 Given a dual functor of ${\mathcal R}$-modules $\mathbb M^*$, by induction on $n$
 $$\aligned {\mathbb Hom}_{\mathcal R} ({\mathbb A} \otimes \overset n\ldots \otimes {\mathbb A}, \mathbb M^*) &
={\mathbb Hom}_{\mathcal R} ({\mathbb A} \otimes \overset{n-1}\ldots \otimes {\mathbb A}, \mathbb Hom_{\mathcal R}(\mathbb A,\mathbb M^*))
\\ & \overset{\text{\ref{3.2}}} =
{\mathbb Hom}_{\mathcal R} ({\mathbb A} \otimes \overset{n-1}\ldots \otimes {\mathbb A}, \mathbb Hom_{\mathcal R}(\mathbb A^{**},\mathbb M^*))\\
& \overset{\text{Ind.Hyp.}}= {\mathbb Hom}_{\mathcal R} ({\mathbb A}^{**} \otimes \overset{n-1}\ldots \otimes {\mathbb A}^{**}, \mathbb Hom_{\mathcal R}(\mathbb A^{**},\mathbb M^*))
\\ & = {\mathbb
Hom}_{\mathcal R} ({\mathbb A}^{**} \otimes \ldots \otimes {\mathbb A}^{**}, \mathbb M^*) .\endaligned$$

Let $i\colon \mathbb A\to \mathbb A^{**}$ be the natural morphism.
The multiplication morphism $m\colon  \mathbb A\otimes  \mathbb A\to  \mathbb A$
defines a unique morphism
$m'\colon  \mathbb A^{**}\otimes  \mathbb A^{**}\to  \mathbb A^{**}$  such that the diagram
$$\xymatrix{  \mathbb A \otimes  \mathbb A \ar[d]^-m \ar[r]^-{i\otimes i} &  \mathbb A^{**}\otimes  \mathbb A^{**}\ar[d]^-{m'}\\  \mathbb A \ar[r]^-i &  \mathbb A^{**}}$$
is commutative, because $\Hom_{\mathcal R}(\mathbb A\otimes \mathbb A,\mathbb A^{**})
=\Hom_{\mathcal R}(\mathbb A^{**}\otimes \mathbb A^{**},\mathbb A^{**})$.
It follows easily that the algebra structure of
${\mathbb A}$ defines an algebra structure on ${\mathbb A}^{**}$.
Let us only check that $m'$ satisfies the associative property: The morphisms
$m'\circ (m'\otimes {\rm Id})$, $m'\circ ({\rm Id}\otimes m')\colon
\mathbb A^{**}\otimes\mathbb A^{**}\otimes \mathbb A^{**}\to \mathbb A^{**}$ are
equal because
$$\aligned & (m'\circ (m'\otimes {\rm Id}))\circ (i\otimes i\otimes i) =
m'\circ (i\otimes i)\circ (m\otimes {\rm Id}) =
i\circ m\circ (m\otimes {\rm Id})\\ & =i\circ m\circ ({\rm Id}\otimes m)=
m'\circ (i\otimes i)\circ ({\rm Id}\otimes m)=
(m'\circ ({\rm Id}\otimes m'))\circ (i\otimes i\otimes i)\endaligned$$

 The kernel of the morphism
$$\Hom_{\mathcal R}(\mathbb A,\mathbb B)\to \Hom_{\mathcal R}(\mathbb A\otimes_{\mathcal R}\mathbb A,\mathbb B),\,\,f\mapsto f\circ m-m\circ (f\otimes f),$$ coincides with the
kernel of the morphism
$$\Hom_{\mathcal R}(\mathbb A^{**},\mathbb B)\to \Hom_{\mathcal R}(\mathbb A^{**}\otimes_{\mathcal R}\mathbb A^{**},\mathbb B),\,\,f\mapsto f\circ m'-m\circ (f\otimes f).$$
Then, ${\rm Hom}_{{\mathcal R}-alg}({\mathbb A},{\mathbb B})={\rm
Hom}_{{\mathcal R}-alg}({\mathbb A}^{**},{\mathbb B})$.

Finally, given a dual functor of $\mathcal R$-modules $\mathbb M^*$, then $\mathbb End_{\mathcal R}\mathbb M^*=(\mathbb M^*\otimes\mathbb M)^*$ is a dual functor of $\mathcal R$-algebras and
$${\rm Hom}_{{\mathcal R}-alg}({\mathbb A},\mathbb End_{\mathcal R}\mathbb M^*)={\rm
Hom}_{{\mathcal R}-alg}({\mathbb A}^{**},\mathbb End_{\mathcal R}\mathbb M^*)$$
Given two $\mathbb
A^{**}$-modules, $\mathbb N$ and $\mathbb M^*$, and a morphism $f\colon \mathbb N\to \mathbb M^*$ of $\mathbb A$-modules, then $f$ is a morphism of ${\mathbb A}^{**}$-modules
because given $n\in\mathbb N$, the morphism $\mathbb A^{**}\to \mathbb M^*$, $a\mapsto
f(an)-af(n)$ is zero because $\mathbb A\to \mathbb M^*$, $a\mapsto
f(an)-af(n)$ is zero. Then, the category of dual functors of ${\mathbb A}$-modules is equal to the category of  dual functors of ${\mathbb A}^{**}$-modules.

\end{proof}

\begin{example} Let $G=\Spec A$ be an $R$-group scheme and let $\mathcal R[G^\cdot]$ be the functor defined
by $\mathcal R[G^\cdot](S)=\{\text{formal sums}\, s_1g_1+\cdots+s_ng_n,$ $n\in{\mathbb N},$ $s_i\in S$ and $g_i\in G^\cdot(S)\}$. It is easy to prove that
$\mathcal R[G^\cdot]^* =\mathcal A$ and $\mathcal R[G^\cdot]^{**}=\mathcal A^*$ (see \cite{Amel2}). Then, the category of (rational) $G$-modules is equivalent to the category of quasi-coherent $\mathcal R[G^\cdot]$-modules, which is equivalent to the category of quasi-coherent $\mathcal A^*$-modules.

\end{example}

\section{$D$-proquasi-coherent modules}

\begin{notation} Let us denote $\mathbb M(\mathcal R)$ the quasi-coherent module associated with
the $R$-module $\mathbb M(R)$, that is, $\mathbb M(\mathcal R)(S):=\mathbb M(R)\otimes_RS$.\end{notation}

There exists a natural morphism $\mathbb M(\mathcal R)\to \mathbb M$, $m\otimes s\mapsto s\cdot m$. Observe that
$$\Hom_{\mathcal R}(\mathcal N,\mathbb M)=\Hom_R(N,\mathbb M(R))=\Hom_{\mathcal R}(\mathcal N,\mathbb M(\mathcal R))$$
for all quasi-coherent modules $\mathcal N$.

\begin{definition} We will say that \label{dualq} a functor of $\mathcal R$-modules $\mathbb M$ is a D-proquasi-coherent module if  the natural morphism $\mathbb M^*\to \mathbb M(\mathcal R)^*$ is injective.\end{definition}

\begin{example} Quasi-coherent modules are D-proquasi-coherent modules, because $\mathcal M(\mathcal R)=\mathcal M$.
\end{example}

\begin{example} \label{3.3.3} If $M=\oplus_I R$ is a free $R$-module, then $\mathcal M^*$ is D-proquasi-coherent: The obvious morphism $\oplus_I \mathcal R\to \prod_I\mathcal R=\mathcal M^*$, factors via $\mathcal M^*(\mathcal R)$. Dually, the composite morphism $\mathcal M=\mathcal M^{**}
\to \mathcal M^*(\mathcal R)^*\to \prod_I \mathcal R$ is injective, then $\mathcal M^{**}
\to \mathcal M^*(\mathcal R)^*$ is injective.\end{example}

\begin{note}  \label{3.3.4} The direct limit of a direct system of functors of D-proquasi-coherent modules is D-proquasi-coherent. Every quotient of a D-proquasi-coherent module is D-proquasi-coherent.\end{note}

\begin{theorem}  \label{W1} A functor of $\mathcal R$-modules $\mathbb M$ D-proquasi-coherent
 if and only if for every $R$-module $N$ the map
$$\Hom_{\mathcal R}(\mathbb M,\mathcal N)\to \Hom_{R}(\mathbb M(R),N),\quad f\mapsto f_R$$
is  injective.\end{theorem}

\begin{proof} If the natural morphism $\mathbb M^*\to \mathbb M(\mathcal R)^*$ is injective, then
$\Hom_{\mathcal R}(\mathbb M,\mathcal S)\subseteq \Hom_{R}(\mathbb M(R),S)$ for all commutative $R$-algebras $S$. Given an $R$-module $N$, consider the $R$-algebra $S:=R\oplus N$, with the multiplication operation $(r,n)\cdot (r',n'):=(rr',rn'+r'n)$. Then,
$$\Hom_{\mathcal R}(\mathbb M, \mathcal R\oplus\mathcal N)=\Hom_{\mathcal R}(\mathbb M,\mathcal S)\subseteq \Hom_{R}(\mathbb M(R),S)= \Hom_{R}(\mathbb M(R), R\oplus N)$$
Hence, $\Hom_{\mathcal R}(\mathbb M,\mathcal N)\subseteq \Hom_{R}(\mathbb M(R),  N)$

Reciprocally,
$\mathbb M^*(S)=\Hom_{\mathcal R}(\mathbb M,\mathcal S)\hookrightarrow
\Hom_{R}(\mathbb M(R),S)=\mathbb M(\mathcal R)^*(S)$ is injective
for all commutative $R$-algebras $S$, hence the morphism $\mathbb M^*\hookrightarrow \mathbb M(\mathcal R)^*$ is injective.
\end{proof}

\begin{corollary} Let $R=K$ be a field. A functor of $\mathcal K$-modules $\mathbb M$ is D-proquasi-coherent
if and only if the natural morphism $\mathbb M^*(K)\to \mathbb M(K)^*:=\Hom_K(\mathbb M(K),K)$, $w\mapsto w_K$ is injective.\end{corollary}

\begin{proof} We  only have  to prove the sufficiency. Let $N=\oplus_I K$ be a $K$-vector space. The diagram
$$\xymatrix @C6pt{\Hom_{\mathcal K}(\mathbb M,\mathcal N) \ar@{^{(}->}[d] \ar[rrr] & & &
\Hom_K(\mathbb M(K),N) \ar@{^{(}->}[d]  \\  \Hom_{\mathcal K}(\mathbb M,\underset{\phantom{I}}{\overset I\prod}\mathcal K) \ar@{=}[r] & \underset{\phantom{I}}{\overset I\prod}\Hom_{\mathcal K}(\mathbb M,\mathcal K) \ar@{^{(}->}[r]  &
\underset{\phantom{I}}{\overset I\prod}\Hom_{K}(\mathbb M(K),K)\ar@{=}[r] & \Hom_{K}(\mathbb M(K),\underset{\phantom{I}}{\overset I\prod} K)}$$
is commutative. Then, the morphism $\Hom_{\mathcal K}(\mathbb M,\mathcal N)\to \Hom_K(\mathbb M(K),N)$ is injective and $\mathbb M$ is D-proquasi-coherent.
\end{proof}

\begin{proposition} \label{QU} If $\mathbb M$ is a D-proquasi-coherent $\mathcal R$-module and $S$ is a commutative $R$-algebra, then the functor of $\mathcal S$-modules $\mathbb M_{|S}$ is a D-proquasi-coherent $\mathcal S$-module.\end{proposition}

\begin{proof} Let $S$ be a commutative $R$-algebra and let $N$ be an $S$-module. The diagram
$$\xymatrix{\Hom_{\mathcal S}(\mathbb M_{|S},\mathcal N) \ar@{=}[r]^-{\text{\ref{adj}}} \ar[d] &
\Hom_{\mathcal R}(\mathbb M,\mathcal N)\ar@{^{(}->}[d]^-{\text{\ref{W1}}}
\\ \Hom_{\mathcal S}(\mathbb M(S),N) \ar[r] & \Hom_{R}(\mathbb M(R),N)}$$
is commutative, then the morphism $\Hom_{\mathcal S}(\mathbb M_{|S},\mathcal N)\to \Hom_{\mathcal S}(\mathbb M(S),N)$ is injective and $\mathbb M_{|S}$ is D-proquasi-coherent.\end{proof}

\begin{lemma} \label{lemaD} A functor of  $\mathcal R$-modules $\mathbb M$ is D-proquasi-coherent  if
and only if the cokernel of every $\mathcal R$-module morphism from $\mathbb M$ to a quasi-coherent module is quasi-coherent, that is, the cokernel of any morphism
$f\colon \mathbb M\to \mathcal N$ is the quasi-coherent module associated with $\Coker f_R$.\end{lemma}

\begin{proof} $\Rightarrow)$ Let $f\colon \mathbb M\to \mathcal N$ be a morphism of $\mathcal R$-modules. Let $N'=\Coker f_R$ and let $\pi\colon \mathcal N\to \mathcal N'$ be the natural epimorphism.  As $(\pi\circ f)_R=0$,   $\pi\circ f=0$ by Theorem \ref{W1}. Then, $\Coker f=\mathcal N'$.

$\Leftarrow)$ Let $f\colon \mathbb M\to \mathcal N$ be a morphism of $\mathcal R$-modules. If $f_R=0$ then $\Coker f=\mathcal N$ and $f=0$. Therefore,
$\mathbb M$ is D-proquasi-coherent, by Theorem \ref{W1}.

\end{proof}

\begin{note} \label{3.14} If $R=K$ is a field, the kernel of every morphism between quasi-coherent modules is quasi-coherent. Then, $\mathbb M$ is D-proquasi-coherent if and only if the image of every morphism $f\colon \mathbb M\to \mathcal N$ for every quasi-coherent module $\mathcal N$, which is the kernel of the morphism $\mathcal N\to \Coker f$,
is a quasi-coherent module.\end{note}

\begin{theorem} \label{P2} Let $R=K$ be a field and  let $\mathbb M$ be a D-proquasi-coherent functor of $\mathcal K$-modules. Let $\{\mathcal M_i\}_{i\in I}$ be the set of  all
quasi-coherent quotients of $\mathbb M$. Then,

$$\mathbb M^*=\ilim{i\in I}  \mathcal M_i^*$$

\end{theorem}

\begin{proof} Let $S$ be a commutative $K$-algebra.
$\mathbb M^*(S)=\Hom_{\mathcal K}(\mathbb M,\mathcal S)$, by Corollary \ref{adj2}.
The morphism $\ilim{i\in I}  \mathcal M_i^*(S)\to \Hom_{\mathcal K}(\mathbb M,\mathcal S)=\mathbb M^*(S)$ is obviously injective, and it is surjective by  Note \ref{3.14}. Hence, $\mathbb M^*=\ilim{i\in I}  \mathcal M_i^*$.

\end{proof}

\begin{corollary} \label{QHom2}Let  $R=K$ be a field.
 If $\mathbb M$ is D-proquasi-coherent, then $\mathbb M^*$ is D-proquasi-coherent.\end{corollary}

\begin{proof} It is a consequence of Theorem \ref{P2}, Example \ref{3.3.3} and Note \ref{3.3.4}.\end{proof}

\begin{definition} A functor of modules $\mathbb M$ is said to be (linearly) separated if
for each commutative $R$-algebra  $S$ and  $m\in\mathbb M(S)$ there exist a commutative $S$-algebra $T$ and a $w\in\mathbb M^*(T)$ such that $w(m)\neq 0$ (that is,
the natural morphism
$\mathbb M\to \mathbb M^{**}$, $m\mapsto \tilde m$, where $\tilde m(w):=w(m)$ for all $w\in\mathbb M^*$, is injective).\end{definition}

Every functor of submodules of a separated functor of modules is separated.

\begin{example} If $\mathbb M$ is a dual functor of modules, then it is separated: Given $0\neq w\in \mathbb M=\mathbb N^*$, there exists an $n\in \mathbb N$ such that $w(n)\neq 0$.
Let $\tilde n\in \mathbb M^*$ be defined by $\tilde n(w'):=w'(n)$, for all $w'\in \mathbb M$. Then $\tilde n(w)\neq 0$.\end{example}

\begin{theorem} \label{QHom} Let $\mathbb M$ be a functor of $\mathcal R$-modules. $\mathbb M$ is D-proquasi-coherent if and only if the morphism
$$\Hom_{\mathcal R}(\mathbb M,\mathbb M')\to \Hom_{R}(\mathbb M(R),\mathbb M'(R)),\quad f\mapsto f_R$$
is injective, for all separated $\mathcal R$-modules, $\mathbb M'$ (such that ${\mathbb M'}^{*}$ are well defined functors).\end{theorem}

\begin{proof} By Theorem \ref{W1}, we only have to prove the necessity. The morphism $\Hom_{\mathcal R}(\mathbb M,\mathbb M')\to
\Hom_{\mathcal R}(\mathbb M'^*,\mathbb M^*)$, $f\mapsto f^*$ is injective:
If $f\neq 0$ there exists an $m\in \mathbb M$ such that $f(m)\neq 0$. Then, there exists a
$w'\in {\mathbb M'}^*$ such that $0\neq w'(f(m))=f^*(w')(m)$. Therefore, $f^*(w')\neq 0$ and $f^*\neq 0$.

From the diagram
$$\xymatrix{\Hom_{\mathcal R}(\mathbb M,\mathbb M') \ar@{^{(}->}[r] \ar[d]
& \Hom_{\mathcal R}(\mathbb M'^*,\mathbb M^*) \ar@{^{(}->}[r]^-{\text{\ref{dualq}}}
&   \Hom_{\mathcal R}(\mathbb M'^*,\mathbb M(\mathcal R)^*)\\
\Hom_{R}(\mathbb M(R),\mathbb M'(R)) \ar@{=}[r] & \Hom_{\mathcal R}(\mathbb M(\mathcal R),\mathbb M')
\ar[ur]&}$$
 one deduces that the morphism $\Hom_{\mathcal R}(\mathbb M,\mathbb M')\to \Hom_{R}(\mathbb M(R),\mathbb M'(R))$ is injective.\end{proof}

\begin{corollary} Let  $R=K$ be a field and let $\mathbb M$, $\mathbb M'$ be D-proquasi-coherent modules, then $\mathbb M\otimes_{\mathcal K}\mathbb M'$ is D-proquasi-coherent.\end{corollary}

\begin{proof} It is due to the inclusion $(\mathbb M\otimes_{\mathcal K}\mathbb M')^*(K)\!  =\!\Hom_{\mathcal K}(\mathbb M\otimes \mathbb M',\mathcal K)\!=\!\Hom_{\mathcal K}(\mathbb M,{\mathbb M'}^*)$ $\overset{\text{\ref{QHom}}}\hookrightarrow \Hom_{K}(\mathbb M(K),{\mathbb M'}^*(K))
\hookrightarrow \Hom_{K}(\mathbb M(K),\mathbb M'(K)^*)=\Hom_{K}(\mathbb M(K)\otimes\mathbb M'(K),K)$ $=(\mathbb M\otimes_{\mathcal K}\mathbb M')(K)^*$.\end{proof}

\begin{proposition} \label{P10}
Let $\mathbb A$ be a  functor of ${\mathcal K}$-algebras and a D-proquasi-coherent module, let $\mathcal M$ and $\mathcal N$ be functors of $\mathbb A$-modules and let $M' \subset M$ be a $K$-vector subspace.  Then,
\begin{enumerate}
\item $\mathcal M'$ is a quasi-coherent
$\mathbb A$-submodule of $\mathcal M$ if and only if $M'$ is an
$\mathbb A(K)$-submodule of $M$.

\item A morphism $f\colon \mathcal M\to \mathcal N$ of functors of $\mathcal K$-modules is a morphism of $\mathbb A$-modules if and only if $f_K\colon M\to N$ is a morphism of $\mathbb A(K)$-modules.

\end{enumerate}

\end{proposition}

\begin{proof} $(1)$
Obviously, if $\mathcal M'$ is an $\mathbb A$-submodule of $\mathcal M$ then $M'$ is
an $\mathbb A(K)$-submodule of $M$. Inversely, let us assume  $M'$ is an $
\mathbb A(K)$-submodule of $M$ and let us consider the natural morphism of
multiplication $\mathbb A \otimes_{\mathcal K} {\mathcal M}' \to {\mathcal M}$. The morphisms $\mathbb A\to \mathcal M$, $a\mapsto a\cdot m'$, for each $m'\in M'$, factors via $\mathcal M'$, then $\mathbb A \otimes_{\mathcal K} {\mathcal M}'\to {\mathcal M}$ factors via $\mathcal M'$.
Therefore, $\mathcal M'$ is a functor of $\mathbb A$-submodules of $\mathcal M$.

$(2)$ The morphism $f$ is a morphism of $\mathbb A$-modules if and only if $F\colon \mathbb A\otimes \mathcal M\to \mathcal N$, $F(a\otimes m):=f(am)-af(m)$ is the zero morphism. Likewise, $f_K$ is a morphism of $\mathbb A(K)$-modules if and only if $F_K\colon \mathbb A(K)\otimes M\to N$, $F_K(a\otimes m)=f_K(am)-af_K(m)$ is the zero morphism. Now, the proposition is a consequence of the inclusions,

$$\aligned \Hom_{\mathcal K}(\mathbb A\otimes \mathcal M,\mathcal N) & =
\Hom_{\mathcal K}(\mathbb A,\mathbb Hom_{\mathcal K}(\mathcal M,\mathcal N))
\underset{\text{\ref{QHom}}}\subseteq \Hom_{K}(\mathbb A(K),\Hom_{\mathcal K}(\mathcal M,\mathcal N))
\\ & \subseteq  \Hom_{K}(\mathbb A(K),\Hom_{K}(M,N))=
  \Hom_{K}(\mathbb A(K)\otimes M,N)\endaligned$$

\end{proof}

\begin{definition} Let $A$ be an $R$-algebra. The associated functor $\mathcal A$ is obviously a functor of $\mathcal R$-algebras. We will say that $\mathcal A$ is a quasi-coherent $\mathcal R$-algebra.\end{definition}

\begin{proposition} \label{invqua6} Let $\mathbb A$ be a  functor of ${\mathcal K}$-algebras and a D-proquasi-coherent module, and  let $B$ be a $K$-algebra.
 Every morphism of $\mathcal K$-algebras $\phi\colon \mathbb A\to \mathcal B$ uniquely factors  through an epimorphism of functors of algebras onto the quasi-coherent algebra associated with $\Ima \phi_K$.\end{proposition}

\begin{proof} By Note \ref{3.14}, the morphism $\phi\colon \mathbb A\to \mathcal B$ uniquely  factors through
an epimorphism $\phi'\colon \mathbb A\to \mathcal B'$, where $B':=\Ima\phi_K$. Obviously $B'$ is a $K$-subalgebra of $B$ and $\phi'$ is a morphism of functors of algebras.\end{proof}

\section{Proquasi-coherent modules}

\begin{definition} A functor of $\mathcal R$-modules is said to be a proquasi-coherent module if it is an inverse limit of quasi-coherent modules.\end{definition}

In this section, $R=K$ will be a field.

\begin{proposition}\label{2125} Let $R=K$ be a field and let $\mathbb M$ be a $\mathcal K$-module
such that $\mathbb M^*$ is well defined.
$\mathbb M$ is separated if and only if the
morphism $\mathbb M \to \bar{\mathbb M}:=(\mathbb M^*(\mathcal K))^*$ is injective.
Therefore, $\mathbb M$ is separated if and only if it is a $\mathcal K$-submodule of a $\mathcal K$-module scheme.
\end{proposition}

\begin{proof}
Assume $\mathbb M$ is separated. Let $m \in \mathbb M(S)$ be such that
$m = 0$ in $\bar{\mathbb M}(S)$. $ \bar{\mathbb M}(S)= {\mathbb M^*(\mathcal K)^*}(S)  {=} {\rm Hom}_K
(\mathbb M^*(K), S)$, then $m(w):= w(m)= 0$ for all $w \in \mathbb M^*(K)$.

Given a commutative $S$-algebra $T$, if one writes $T = \oplus_{i\in I} K \cdot e_i$,
one notices that $$\mathbb M^*(T)
\overset{\text{\ref{adj2}}}= {\rm Hom}_{\mathcal K} (\mathbb M, {\mathcal T}) = {\rm Hom}_{\mathcal K} (\mathbb M,
\oplus_I{{\mathcal K}}) \subset \prod_I {\rm Hom}_{\mathcal K} (\mathbb M, {{\mathcal K}})$$ which
assigns to every $w_T \in \mathbb M^*(T)$ a $(w_i) \in \prod \mathbb M^*(K)$.
Explicitly, given $m' \in \mathbb M(T)$, then $w_T(m')= \sum_i w_i(m') \cdot
e_i$. Therefore, $w_T(m)=0$ for all $w_T \in \mathbb M^*(T)$. As $\mathbb M$ is separated, this means that
$m = 0$, i.e., the morphism $\mathbb M \to \bar{\mathbb M}$ is injective.

Now, assume $\mathbb M \to \bar{\mathbb M}$ is injective. Observe that $\bar{\mathbb M}$ is separated because is reflexive. Then $\mathbb M$ is separated.

Finally, the second statement of the proposition is obvious.
\end{proof}

\begin{proposition} \label{P6} If $\mathbb M$ is a  proquasi-coherent $\mathcal K$-module then it is a dual $\mathcal K$-module and it is a direct  limit of $\mathcal K$-shemes of modules.
In particular, proquasi-coherent modules are D-proquasi-coherent modules.
\end{proposition}

\begin{proof} $\mathbb M=\plim{} \mathcal M_i=(\ilim{} \mathcal M_i^*)^*$. As $\ilim{} \mathcal M_i^*$ is D-proquasi-coherent, its dual, which is $\mathbb M$, is a direct limit of $\mathcal K$-module schemes, by Theorem \ref{P2}.\end{proof}

\begin{theorem} \label{RPQ} Let $R=K$ be a field. $\mathbb M$ is a reflexive functor of $\mathcal K$-modules if and only if
$\mathbb M$ is equal to the inverse limit of its quasi-coherent quotients.
In particular, reflexive functors of $\mathcal K$-modules are proquasi-coherent.\end{theorem}

\begin{proof} Suppose that $\mathbb M$ is reflexive. $\mathbb M^*$ is separated, because it is a dual functor of modules.  By Proposition \ref{2125}, the morphism
$\mathbb M^*\to \mathbb M(\mathcal K)^* $ is injective.
Then, $\mathbb M$ is D-proquasi-coherent.
Let $\{\mathcal M_i\}_{i\in I}$ be the set of  all
quasi-coherent quotients of $\mathbb M$. Then,
$\mathbb M^*=\ilim{i\in I}  \mathcal M_i^*$, by Theorem \ref{P2}. Therefore,
$\mathbb M=\mathbb M^{**}=\plim{i\in I}  \mathcal M_i$.

Suppose now that $\mathbb M$ is equal to the inverse limit of its quasi-coherent quotients. By Proposition \ref{P6}, $\mathbb M$ is D-proquasi-coherent.
By Theorem \ref{P2}, $\mathbb M=\mathbb M^{**}$.

\end{proof}

Let $R={\mathbb Z}$ and $M={\mathbb Z}/2{\mathbb Z}$. Then, $\mathbb M:=\mathcal M^*$ is reflexive but it is not D-proquasi-coherent, because $\mathbb M(\mathcal R)^*=0$,  since $\mathbb M(R)=0$.

\begin{proposition} \label{K1} Let $f\colon \mathbb P\to \mathbb M$ be a morphism of functors of $\mathcal K$-modules.
If $\mathbb P$ is proquasi-coherent and $\mathbb M$ is separated then
$\mathbb Ker\, f$ is proquasi-coherent.\end{proposition}

\begin{proof} Let $V$ be a $K$-vector space such that there exists an injective morphism
$\mathbb M\hookrightarrow \mathcal V^*$. We can assume $\mathbb M=\mathcal V^*= \prod_{I}\mathcal K$. Given $I'\subset I$, let $f_{I'}$ be the composition of
$f$ with the obvious projection $ \prod_{I}\mathcal K\to  \prod_{I'}\mathcal K$.
Then
$$\mathbb Ker\, f=\plim{I'\subset I,\, \#I'<\infty} \mathbb Ker\, f_{I'}$$
It is sufficient to prove that $\mathbb Ker\, f_{I'}$ is proquasi-coherent, since the inverse limit of proquasi-coherent modules is proquasi-coherent. As $\# I'<\infty$ it is sufficient to prove that the kernel of every morphism $f\colon \mathbb P \to \mathcal K$ is proquasi-coherent.

If  $f\colon \mathbb P\to\mathcal K$ is the zero morphism the proposition  is obvious. Assume $f\neq 0$.
Then, $f$ is surjective. Let us write $\mathbb P= \plim{i} \mathcal V_i$ and let $v=(v_i)\in \plim{i} V_i=\mathbb P(K)$ be a vector such that $f_K((v_i))\neq 0$. Then $\mathbb P=\mathbb Ker\, f\oplus\mathcal K\cdot v$.
Let $\bar V_i:=V_i/\langle v_i\rangle$.
Let us prove that $\mathbb Ker\, f=\plim{i}\bar{\mathcal V}_i$: Let $i'$ be such that $v_{i'}\neq 0$. Consider the exact sequences
$$0\to\mathcal K \cdot v_i\to \mathcal V_i\to \bar{\mathcal V}_i\to 0, \qquad (i>i')$$
Dually, we have the exact sequences
$$0\to \bar{\mathcal V}_i^*\to \mathcal V_i^* \to \mathcal K\to 0$$
Taking the direct limit we have the split exact sequence
$$0\to \ilim{i} (\bar{\mathcal V}_i^*)\to \ilim{i} (\mathcal V_i^*) \to \mathcal K\to 0$$
Dually, we have the exact sequence
$$0\to \mathcal K\cdot v\to \mathbb P\to \plim{i} \bar{\mathcal V}_i\to 0$$
Then, $\mathbb Ker\, f\to \plim{i}\bar{\mathcal V}_i$, $(v_i)_i\mapsto (\bar v_i)_i$ is an isomorphism.

\end{proof}

\begin{corollary} Every direct summand of a proquasi-coherent module is proquasi-coherent.\end{corollary}

\begin{theorem} \label{P=DQ} Let $\mathbb M$ be a functor of $\mathcal K$-modules. $\mathbb M$ is  proquasi-coherent
if and only if  $\mathbb M$ is a dual functor of $\mathcal K$-modules and it is D-proquasi-coherent. \end{theorem}

\begin{proof} By Proposition \ref{P6}, we  only have  to prove the sufficiency.
 Let us write $\mathbb M=\mathbb N^*$. The dual morphism of the natural
 morphism $\mathbb N\to \mathbb N^{**}$ is a retraction of the natural morphism $\mathbb M\to \mathbb M^{**}$. Then, $\mathbb M^{**}=\mathbb M\oplus\mathbb M'$. $\mathbb M$ is proquasi-coherent, because $\mathbb M^{**}$ is proquasi-coherent, by Theorem \ref{P2}.

\end{proof}

\begin{corollary} \label{P=DQ2} A functor of $\mathcal K$-modules is proquasi-coherent if and only if it is the dual functor of $\mathcal K$-modules of a D-proquasi-coherent module.\end{corollary}

\begin{proof} If $\mathbb M=\plim{i}\mathcal M_i$ is proquasi-coherent, then $\mathbb M=(\ilim{i} \mathcal M_i^*)^*$.
$\ilim{i} \mathcal M_i^*$ is D-proquasi-coherent and $\mathbb M=(\ilim{i} \mathcal M_i^*)^*$.
If $\mathbb M'$ is D-proquasi-coherent,
then
$\mathbb M'^*$ is D-proquasi-coherent, by Corollary \ref{QHom2}.
By Theorem \ref{P=DQ}, $\mathbb M'^*$ is proquasi-coherent.\end{proof}

\begin{proposition} \label{main} Let $M$ be an $R$-module. Then,
$$\mathbb Hom_{\mathcal R}(\prod_I\mathcal R,\mathcal M)=\oplus_I \mathbb Hom_{\mathcal R}(\mathcal R,\mathcal M)=\oplus_I\mathcal M$$

\end{proposition}

\begin{proof} $\mathbb Hom_{\mathcal R}(\prod_I\mathcal R,\mathcal M)=\mathbb Hom_{\mathcal R}((\oplus_I\mathcal R)^*,\mathcal M)\overset{\text{\ref{1.8Amel}}}=(\oplus_I\mathcal R)\otimes \mathcal M=\oplus_I\mathcal M$.\end{proof}

\begin{proposition} \label{Fcansado} Let $I$ be a totally ordered set and  $\{f_{ij}\colon M_i\to M_j\}_{i\geq j\in I}$ an inverse system of $K$-modules. Then, $\plim{i} \mathcal M_i$ is reflexive.
\end{proposition}

\begin{proof}
$\plim{i} \mathcal M_i$ is a direct limit of submodule schemes  $\mathcal V_j^*$, by \ref{P2} and \ref{P=DQ2}.
If all the vector spaces $V_j$ are finite dimensional then
$\plim{i} \mathcal M_i$ is quasi-coherent, then it is reflexive. In other case,
there exists an injective morphism  $f\colon \prod_{\mathbb N}\mathcal K\hookrightarrow \plim{i} \mathcal M_i$. Let $\pi_j\colon \plim{i} \mathcal M_i\to \mathcal M_j$ be the natural morphisms. Let $g_r\colon \mathcal K^r\hookrightarrow \prod_{\mathbb N}\mathcal K$ be
defined by $g_r(\lambda_1,\cdots,\lambda_r):=(\lambda_1,\cdots,\lambda_r,0,\cdots,0,\cdots)$.
Let  $i_1\in I$ be such that $\pi_{i_1}\circ f\circ g_1$ is injective. Recursively,
let $i_n>i_{n-1}$ be such that $\pi_{i_n}\circ f\circ g_n$ is injective.
If there exists a $j>i_n$ for all $n$, the composite morphism $\oplus_{\mathbb N}\mathcal K\subset\prod_{\mathbb N}\mathcal K \to \mathcal M_j$ is injective,
and by Proposition \ref{main} the morphism  $\prod_{\mathbb N}\mathcal K \to \mathcal M_j$ factors through the projection onto a $\mathcal K^r$, which is contradictory. In conclusion, $\plim{i} \mathcal M_i=\plim{n\in\mathbb N} \mathcal M_{i_n}$.

 Let $\mathcal M'_{i_r}$ be the image of  $\plim{n} \mathcal M_{i_n}$ in $\mathcal M_{i_r}$. Then, $\plim{n} \mathcal M'_{i_n}=\plim{n} \mathcal M_{i_n}$.
 Let $H_n:=\Ker [M'_{i_n}\to M'_{i_{n-1}}]$. Then, $\plim{n} \mathcal M_{i_n}\simeq\prod_n \mathcal H_n$. By Lemma \ref{uf}, $\plim{n} \mathcal M_{i_n}$ is reflexive.

\end{proof}

\begin{note} We do not know if every  proquasi-coherent functor of $\mathcal K$-modules is  reflexive.
\end{note}

\begin{proposition} \label{P4} If $\mathbb P,\mathbb P'$ are proquasi-coherent $\mathcal K$-modules, then  $\mathbb Hom_{\mathcal K}(\mathbb P, \mathbb P')$ is proquasi-coherent. In particular, $\mathbb P^*$ and $(\mathbb P\otimes \mathbb P')^*$ are proquasi-coherent.

\end{proposition}

\begin{proof} Let us write $\mathbb P=\ilim{i} \mathcal V_i^*$ and $\mathbb P'=\plim{j} \mathcal V'_j$.
Then,

$$\mathbb Hom_{\mathcal K}(\mathbb P, \mathbb P')=\mathbb Hom_{\mathcal K}(\ilim{i} \mathcal V_i^*, \plim{j} \mathcal V'_j)=\plim{i,j} \mathbb Hom_{\mathcal K}( \mathcal V_i^*,  \mathcal V'_j)=\plim{i,j} (\mathcal V_i\otimes\mathcal V'_j)$$
Hence, $\mathbb Hom(\mathbb P, \mathbb P')$ is proquasi-coherent.

\end{proof}

\begin{proposition} \label{P9} Let $\mathbb A$ be a functor of  $\mathcal K$-algebras  and a D-proquasi-coherent module,
and let $\mathbb P,\mathbb P'$ be proquasi-coherent $\mathcal K$-modules and $\mathbb A$-modules.
Then, a morphism of $\mathcal K$-modules, $f\colon \mathbb P\to\mathbb P'$, is a morphism
of $\mathbb A$-modules if and only if $f_K\colon \mathbb P(K)\to\mathbb P'(K)$ is a morphism of $\mathbb A(K)$-modules.

\end{proposition}

\begin{proof} Proceed  as in the proof of Proposition \ref{P10} (2).

\end{proof}

\section{A family $\mathfrak F$ of reflexive functors of $\mathcal R$-modules}

\label{F}

Consider $\prod_{j\in J} \mathcal R$ as a functor of $\mathcal R$-algebras ($(\lambda_j)_j\cdot (\mu_j)_j:=(\lambda_j\cdot \mu_j)_j$). If $\{\mathbb M_j\}_{j\in J}$ is a set of $\mathcal R$-modules, then $\oplus_{j\in J} \mathbb M_j$ and
$\prod_{j\in J}\mathbb M_j$ are naturally functors of $\prod_{j\in J} \mathcal R$-modules.

\begin{lemma} \label{uf} Let $\mathbb M$ be a dual functor of $\mathcal R$-modules and a functor of $\prod_{j\in J} \mathcal R$-modules. If there exist a
set of reflexive functors of $\mathcal R$-modules, $\{\mathbb M_j\}_{j\in J}$, and inclusions of $\prod_{j\in J} \mathcal R$-modules
$$\oplus_{j\in J} \mathbb M_j\subseteq \mathbb M\subseteq \prod_{j\in J}\mathbb M_j$$
(where $\oplus_{j\in J} \mathbb M_j \subseteq \prod_{j\in J}\mathbb M_j$ is the obvious inclusion)
then, \begin{enumerate} \item $\mathbb M$ is a reflexive functor of $\mathcal R$-modules.

\item For every $R$-module $N$ we have
$$\oplus_{j\in J} \mathbb Hom_{\mathcal R}(\mathbb M_j,\mathcal N)\subseteq \mathbb Hom_{\mathcal R}(\mathbb M,\mathcal N)\subseteq \prod_{j\in J}\mathbb Hom_{\mathcal R}(\mathbb M_j,\mathcal N)$$
\end{enumerate}

\end{lemma}

\begin{proof} Given $w\in \mathbb Hom_{\mathcal R}(\mathbb M,\mathcal N)$, let $w_j:=w_{|\mathbb M_j}$, for all
$j\in J$. Given $m=(m_j)_j\in \mathbb M\subset \prod_j\mathbb M_j$, then $w_j(m_j)=0$ for all $j\in J$, except for a finite subset $I\subset J$, and $w(m)=\sum_{i\in I}w_{i}(m_{i})$:
Let
$W\colon \prod_j\mathcal R\to \mathcal N$ be defined by $W((\lambda_j)_j):=w(( \lambda_jm_j)_j)$. By Proposition \ref{main},
there exists a finite subset $I\subset J$ such that
$W((\lambda_j)_j)=W((\lambda_{i})_{i\in I})$. Hence,
$w_j(m_j)=w(m_j)=0$ for all $j\in J\backslash I$, and $w(m)=w((m_j)_j)=w( (m_i)_{i\in I})=\sum_{i\in I}w_{i}(m_{i})$.

Then, $\oplus_{j\in J} \mathbb Hom_{\mathcal R}(\mathbb M_j,\mathcal N)\subseteq \mathbb Hom_{\mathcal R}(\mathbb M,\mathcal N)\subseteq \prod_{j\in J}\mathbb Hom_{\mathcal R}(\mathbb M_j,\mathcal N)$.

In particular, we have
$$\oplus_j \mathbb M_j^*\subseteq \mathbb M^*\subseteq\prod_j \mathbb M_j^*$$

$\mathbb M^*$ is a $\prod_j \mathcal R$-module and again
$$\oplus_j \mathbb M_j\subseteq \mathbb M^{**}\subseteq \prod_j\mathbb M_j$$
We have $\oplus_j\mathbb M_j\subseteq \mathbb M\subseteq \mathbb M^{**}\subseteq \prod_j\mathbb M_j$, and again
$(\mathbb M^{**})^{*}\subseteq \mathbb M^{*}\subseteq \prod_j\mathbb M_j^*$.
The natural morphism $(\mathbb M^{**})^*\to \mathbb M^{*}$ is an epimorphism, because the natural morphism  $\mathbb M^*\to (\mathbb M^*)^{**}$ is a section. Therefore,
$\mathbb M^*=\mathbb M^{***}$.

The inclusion $\mathbb M\subseteq \mathbb M^{**}$ has a retraction, because
$\mathbb M=\mathbb M'^*$ is a dual functor and the natural morphism
$\mathbb M'^*\to (\mathbb M'^*)^{**}$ has a retraction. Then, $\mathbb M^{**}=\mathbb M\oplus \mathbb M''$.
Dually, $\mathbb M^*=\mathbb M^{***}$, so ${\mathbb M''}^*=0$. Hence, $\mathbb M''=0$, because $\mathbb M''\subset \prod_j \mathbb M_j$ and for any $0\neq (m_j)\in \prod_j \mathbb M_j$ there exist a $j\in J$ and a $w_j\in \mathbb M_j^*$ such that  $w_j(m_j)\neq 0$ (recall that the functor of modules $\mathbb M_j$ are reflexive).
Therefore, $\mathbb M=\mathbb M^{**}$.

\end{proof}

 \begin{definition} \label{defF} Let $\mathfrak F$ be the family of dual functors of $\mathcal R$-modules, $\mathbb M$, such that there exist  a set $J$ (which depends on $\mathbb M$), a structure of functor of $\prod_J \mathcal R$-modules on $\mathbb M$  and inclusions of functors of $\prod_J \mathcal R$-modules
$$\oplus_J\mathcal R\subseteq \mathbb M\subseteq \prod_J \mathcal R$$
(where $\oplus_{j\in J} \mathbb M_j \subseteq \prod_{j\in J}\mathbb M_j$ is the obvious inclusion).
\end{definition}

\begin{note} Every $\mathbb M\in\mathfrak F$ is reflexive, by Lemma \ref{uf}.\end{note}

\begin{examples} If $V$ is a free $\mathcal R$-module, $\mathcal V,\mathcal V^*\in\mathfrak F$. If we have a set $\{\mathbb M_i\in\mathfrak F\}_{i\in I}$, then
 $\oplus_{i\in I}\mathbb M_i, \prod_{i\in I}\mathbb M_i\in\mathfrak F$, as it is easy to check.

Let $V=\oplus_I R$ and $W=\oplus_J R$ be free $R$-modules, then
$\mathbb Hom_{\mathcal R}(\mathcal V,\mathcal W)\in\mathfrak F$: we have the obvious inclusions of  functors of $\prod_{I\times J} \mathcal R$-modules,
$\oplus_{I\times J}\mathcal R\subseteq \mathbb Hom_{\mathcal R}(\mathcal V,\mathcal W)=
\prod_I(\oplus_J\mathcal R)\subseteq \prod_{I\times J}\mathcal R$. This example has motivated Definition \ref{defF}.

\end{examples}

\begin{note} A quasi-coherent module $\mathcal M\in\mathfrak F$ if and only if
$M$ is a free $R$-module: Let us write $\oplus_J\mathcal R\subseteq \mathcal M\subseteq \prod_J \mathcal R$. Let $N=M/\oplus_J R$ and let $\pi\colon M\to N$ be the quotient morphism. By Lemma \ref{uf}, the morphism $\Hom_R(M,N)\to \Hom_R(\oplus_J R, N)$ is injective. As $\pi\mapsto 0$, then $\pi=0$ and $M=\oplus_JR$.

Let $\mathfrak F'$ be the family of functors of $\mathcal R$-modules, $\mathbb M$, such that $\mathbb M$ is a direct summand of some  functor of $\mathcal R$-modules of $\mathfrak F$.
Obviously, $\mathcal M\in\mathfrak F'$ if $M$ is a projective $R$-module.
All other  results obtained in this section for $\mathfrak F$ remain valid for $\mathfrak F'$.
\end{note}

\begin{proposition} If $\mathbb M'$ is a reflexive functor of $\mathcal R$-modules and $\mathbb M\in \mathfrak F$, then $\mathbb Hom_{\mathcal R}(\mathbb M',\mathbb M)$ is reflexive.\end{proposition}

\begin{proof} $\mathbb Hom_{\mathcal R}(\mathbb M',\mathbb M)=(\mathbb M'\otimes\mathbb M^*)^*$ is a dual functor. Following previous notations, since
$\oplus_J\mathcal R\subseteq \mathbb M\subseteq \prod_J \mathcal R$, we have the inclusions
\begin{equation} \label{EQ1} \oplus_J\mathbb M'^*\subseteq \mathbb Hom_{\mathcal R}(\mathbb M',\oplus_J\mathcal R)\subseteq \mathbb Hom_{\mathcal R}(\mathbb M',\mathbb M)\subseteq
\mathbb Hom_{\mathcal R}(\mathbb M',\prod_J\mathcal R)=\prod_J \mathbb M'^*\end{equation}
Therefore, by Lemma \ref{uf}, $\mathbb Hom_{\mathcal R}(\mathbb M',\mathbb M)$ is reflexive.
\end{proof}

\begin{proposition} If $\mathbb M'$ is a reflexive functor of $\mathcal R$-modules and $\mathbb M\in \mathfrak F$, then $\mathbb Hom_{\mathcal R}(\mathbb M,\mathbb M')$ is reflexive.\end{proposition}

\begin{proof}
$\mathbb Hom_{\mathcal R}(\mathbb M,\mathbb M')=(\mathbb M\otimes\mathbb M'^*)^*$ is a dual functor.  Let $0\neq f\in \mathbb Hom_{\mathcal R}(\mathbb M,\mathbb M')$. There exists $w\in \mathbb M'^*$ such that $w\circ f\neq 0$. Let us follow previous notations. By Lemma \ref{uf}, $(w\circ f)_{|\oplus_J\mathcal R}\neq 0$, then $f_{|\oplus_J\mathcal R}\neq 0$.
Therefore, $\mathbb Hom_{\mathcal R}(\oplus_J\mathcal R, \mathbb M')\supseteq
\mathbb Hom_{\mathcal R}(\mathbb M, \mathbb M')$ and, likewise, $\mathbb Hom_{\mathcal R}(\oplus_J\mathcal R, \mathbb M')\supseteq \mathbb Hom_{\mathcal R}(\prod_J\mathcal R, \mathbb M')$. Hence, we have the inclusions
$$\prod_J \mathbb M'=\mathbb Hom_{\mathcal R}(\oplus_J\mathcal R, \mathbb M')\supseteq
\mathbb Hom_{\mathcal R}(\mathbb M, \mathbb M')\supseteq \mathbb Hom_{\mathcal R}(\prod_J\mathcal R, \mathbb M')\supseteq \oplus_J \mathbb M'$$
Therefore, by Lemma \ref{uf}, $\mathbb Hom_{\mathcal R}(\mathbb M,\mathbb M')$ is reflexive.
\end{proof}

\begin{theorem} \label{F8.7}
If $\mathbb M',\mathbb M\in\mathfrak F$, then $\mathbb Hom_{\mathcal R}(\mathbb M',\mathbb M)\in \mathfrak F$.\end{theorem}

 \begin{proof} Let us write $\oplus_I\mathcal R\subseteq \mathbb M'\subseteq \prod_I \mathcal R$ and $\oplus_J\mathcal R\subseteq \mathbb M\subseteq \prod_J \mathcal R$. By Lemma \ref{uf}, $\oplus_I\mathcal R\subseteq \mathbb M'^*\subseteq \prod_I \mathcal R$, then by Equation \ref{EQ1},
\begin{equation}\label{EQ2}\oplus_{I\times J}\mathcal R\subseteq \oplus_{J}\mathbb M'^*\subseteq \mathbb Hom_{\mathcal R}(\mathbb M',\mathbb M)\subseteq \prod_{J}\mathbb M'^*\subseteq\prod_{I\times J} \mathcal R\end{equation}
Observe that $$(\prod_I\mathcal R\otimes \prod_J\mathcal R)^{*}=\mathbb Hom_{\mathcal R}(
\prod_I\mathcal R\otimes \prod_J\mathcal R,\mathcal R)=\mathbb Hom_{\mathcal R}(
\prod_I\mathcal R, \oplus_J\mathcal R)\overset{\text{\ref{prop4} }}=(\oplus_I\mathcal R)\otimes(\oplus_J\mathcal R)$$

$\mathbb Hom_{\mathcal R}(\mathbb M',\mathbb M)=(\mathbb M'\otimes\mathbb M^*)^*$ is a $\prod_I\mathcal R\otimes \prod_J\mathcal R$-module,
then by Proposition \ref{2.4}, it is a $(\prod_I\mathcal R\otimes_{\mathcal R} \prod_J\mathcal R)^{**}=(\oplus_I\mathcal R\otimes_{\mathcal R} \oplus_J\mathcal R)^*=\prod_{I\times J}\mathcal R$-module. Finally,
by Equation \ref{EQ2}, $\mathbb Hom_{\mathcal R}(\mathbb M',\mathbb M)\in \mathfrak F$.
\end{proof}

 \begin{theorem} If $\mathbb M,\mathbb M'\in\mathfrak F$, then $(\mathbb M'\otimes\mathbb M)^{**}\in\mathfrak F$ and it is the closure of dual functors of $\mathcal R$-modules of $\mathbb M'\otimes\mathbb M$.\end{theorem}

\begin{proof} As $\mathbb M^*\in\mathfrak F$, we have $(\mathbb M'\otimes\mathbb M)^*=\mathbb Hom_{\mathcal R}(\mathbb M',\mathbb M^*)\in \mathfrak F$. Hence, firstly  $(\mathbb M'\otimes\mathbb M)^*$ is reflexive and by Proposition \ref{3.2} the closure of dual functors of $\mathbb M'\otimes\mathbb M$ is $(\mathbb M'\otimes\mathbb M)^{**}$, secondly
$(\mathbb M'\otimes\mathbb M)^{**}\in\mathfrak F$.
\end{proof}

\begin{proposition} \label{cansado} Let $R=K$ be a field.
Let $I$  be a totally ordered set and $\{f_{ij}\colon M_i\to M_j\}_{i\geq j\in I}$ be an inverse system of morphisms of $K$-modules. Then, $\plim{i\in I}\mathcal M_i\in \mathfrak F$.\end{proposition}

\begin{proof} $\plim{i\in I}\mathcal M_i$ is a direct product of $\mathcal K$-quasi-coherent modules, by the proof of Proposition \ref{Fcansado}. Hence, $\plim{i\in I}\mathcal M_i\in\mathfrak F$.
\end{proof}

\begin{proposition} \label{F8.10} If $\mathbb M\in\mathfrak F$, then $\mathbb M$ is D-proquasi-coherent. \end{proposition}

 \begin{proof} We have to prove that the  morphism
$$\Hom_{\mathcal R}(\mathbb M,\mathcal N) \subseteq \Hom_R(\mathbb M(R),N),\,\, w\mapsto w_R$$
is injective, for all $R$-modules $N$.
 Let us follow previous notations. By Lemma \ref{uf}, $w\in \Hom_{\mathcal R}(\mathbb M,\mathcal N)$ is determined by $w_{|\oplus_J\mathcal R}$, and this one is determined by $(w_{|\oplus_J\mathcal R})_R$. As $\oplus_J R\subseteq  \mathbb M(R)$,  $w$ is determined by $w_R$.
 \end{proof}

\begin{lemma} \label{invqua} Let $\mathbb M\in\mathfrak F$ and let $N$ be an $R$-module. Then, every morphism of $\mathcal R$-modules $\phi\colon \mathbb M\to \mathcal N$ uniquely factors through an epimorphism onto the quasi-coherent module associated with the $R$-submodule of $N$, $\Ima \phi_R\subseteq N$.\end{lemma}

\begin{proof} Let us follows previous notations. Consider $\oplus_J\mathcal R\subseteq\mathbb M\subseteq \prod_J\mathcal R$. By Lemma \ref{uf},
$\mathbb M^*\subseteq (\oplus_J\mathcal R)^*=\prod_J\mathcal R$.
 The morphism $\phi_{|\oplus_J\mathcal R}\colon \oplus_J\mathcal R\to \mathcal N$ factors via the quasi-coherent module associated with $N':=\Ima (\phi_{|\oplus_J\mathcal R})_R$. Then, the dual morphism $\phi^*\colon \mathcal N^*\to \mathbb M^*\subseteq \prod_J\mathcal R$, factors via, ${\mathcal N'}^*$. Hence, $\phi$ factors via a morphism $\phi'\colon \mathbb M\to \mathcal N'$.
 In particular, $\phi'$ is an epimorphism and $N'=\Ima \phi_R$.

Uniqueness: Assume $\phi$ factors through an epimorphism $\phi''\colon \mathbb M\to \mathcal N'$. The morphisms $\phi',\phi''$ are determined by  ${\phi'}_R,{\phi''}_R$. Then, $\phi'=\phi''$, because ${\phi'}_R={\phi''}_R$.

\end{proof}

\begin{theorem} \label{FP2} Let $\mathbb M\in\mathfrak F$. Let $\{\mathcal M_i\}_{i\in I}$ be the set of all
quasi-coherent quotients of $\mathbb M$. Then,
$\mathbb M^*=\ilim{i\in I}  \mathcal M_i^*$.
Therefore, $$\mathbb M=\plim{i\in I}  \mathcal M_i.$$

\end{theorem}

\begin{proof} Proceed  as in the proof of Theorem \ref{P2}.\end{proof}

\begin{proposition} Let $\mathbb A\in\mathfrak F$ be a functor of  $\mathcal R$-algebras
and let $\mathbb M,\mathbb M'\in\mathfrak F$ be functors of $\mathbb A$-modules.
Then, a morphism of $\mathcal R$-modules, $f\colon \mathbb M\to\mathbb M'$, is a morphism
of $\mathbb A$-modules if and only if $f_R\colon \mathbb M(R)\to\mathbb M'(R)$ is a morphism of $\mathbb A(R)$-modules.

\end{proposition}

\begin{proof} Proceed as in Proposition \ref{P9}.

\end{proof}

\begin{notation} Let $M$ be an $R$-module and let $M'\subseteq M$ be an $R$-submodule. By abuse of notation we will say that $\mathcal M'$ is a quasi-coherent submodule of $\mathcal M$.\end{notation}

\begin{proposition} \label{invqua4}
Let $\mathbb A\in\mathfrak F$ be a  functor of ${\mathcal R}$-algebras, let $\mathcal M$ be an $\mathbb A$-module and let $M' \subset M$ be an $R$-submodule.  Then, $\mathcal M'$ is a quasi-coherent
$\mathbb A$-submodule of $\mathcal M$ if and only if $M'$ is an
$\mathbb A(R)$-submodule of $M$.
\end{proposition}

\begin{proof}
Obviously, if $\mathcal M'$ is an $\mathbb A$-submodule of $\mathcal M$ then $M'$ is
an $\mathbb A(R)$-submodule of $M$. Inversely, let us assume  $M'$ is an $
\mathbb A(R)$-submodule of $M$ and let us consider the natural morphism of
multiplication $\mathbb A \otimes_{\mathcal R} {\mathcal M}' \to {\mathcal M}$. By  Lemma \ref{invqua}, the morphisms $\mathbb A\to \mathcal M$, $a\mapsto a\cdot m'$, for each $m'\in M'$, uniquely factors via $\mathcal M'$. Let $\oplus_IR\overset q\to\oplus_JR\overset p\to M'\to 0$ be an exact sequence. Let $i$ be the morphism $\mathcal M'\to \mathcal M$. There exists a (unique) morphism $f'$ such that the diagram
$$\xymatrix{ \mathbb A\otimes_{\mathcal R} (\oplus_I \mathcal R)
\ar[r]^-{Id\otimes q} & \mathbb A\otimes_{\mathcal R} (\oplus_J \mathcal  R) \ar[r]^-{Id\otimes p} \ar[d]^-{f'} & \mathbb A\otimes_{\mathcal R} \mathcal M' \ar[r] \ar[d] & 0\\ & \mathcal M'\ar[r]_-i & \mathcal M & }$$
is commutative. Since $i\circ f'\circ (Id\otimes q)=0$, then $f'\circ (Id\otimes q)=0$.
Hence, $\mathbb A \otimes_{\mathcal R} {\mathcal M}'\to {\mathcal M}$ factors through $\mathcal M'$.

$F\colon \mathbb A\otimes_{\mathcal R}\mathbb A\to \mathcal M'$, $F(a\otimes a'):=
a(a'm')-(aa')m'$ (for any $m'\in\mathcal M'$) is the zero morphism:
$F$ lifts to a (unique) morphism $\bar F\colon (\mathbb A\otimes_{\mathcal R}\mathbb A)^{**}\to\mathcal M'$. Observe that $i\circ\bar F=0$ because $i\circ F=0$, then $\bar F_R=0$ because
$i_R$ is injective. Finally, $\bar F=0$ because it is determined by $\bar F_R$; and $F=0$.
Likewise, $1\cdot m'=m'$, for all $m'\in\mathcal M'$.

In conclusion, $\mathcal M'$
is a quasi-coherent ${\mathbb A}$-submodule of $\mathcal M$.
\end{proof}

\begin{proposition} \label{invqua2} Let $\mathbb A\in\mathfrak F$ and $\mathbb B$ be functors of $\mathcal R$-algebras and assume that there exists an injective morphism of $\mathcal R$ modules $\mathbb B\hookrightarrow \mathcal N$. Then, any morphism of $\mathcal R$-algebras $\phi\colon \mathbb A\to \mathbb B$ uniquely factors  through an epimorphism of algebras onto the quasi-coherent algebra associated with $\Ima \phi_R$, that is to say, $(\Ima \phi)(\mathcal R)$.\end{proposition}

\begin{proof} By Lemma \ref{invqua}, the morphism $\phi\colon \mathbb A\to \mathbb B$ uniquely  factors through
an epimorphism $\phi'\colon \mathbb A\to \mathcal B'$, where $B':=\Ima\phi_R$. Obviously $B'$ is an $R$-subalgebra of $\mathbb B(R)$. We have to check that $\phi'$ is a morphism of functors of algebras.

Observe that if a morphism $f\colon \mathbb A\otimes \mathbb A\to \mathcal N$ factors  through an epimorphism onto a quasi-coherent  submodule $\mathcal N'$ of $\mathcal N$ then uniquely  factors  through $\mathcal N'$, because $f$ and any morphism on $\mathcal N'$ uniquely  factors through
$(\mathbb A\otimes \mathbb A)^{**}\in\mathfrak F$.

Consider the diagram

$$\xymatrix{\mathbb A\otimes\mathbb A \ar[d]^-{m_{\mathbb A}} \ar[r]^-{\phi'\otimes\phi'}
& \mathcal B'\otimes\mathcal B' \ar[d]^-{m_{\mathcal B'}}\ar[r]^-{i\otimes i}
& \mathbb B\otimes\mathbb B \ar[d]^-{m_{\mathbb B}} & \\ \mathbb A \ar[r]^-{\phi'} & \mathcal B' \ar[r]^-i & \mathbb B \ar@{^{(}->}[r] & \mathcal N,}$$
where $m_{\mathbb A}, m_{\mathcal B'}$ and $m_{\mathbb B}$ are the multiplication morphisms and $i$ is the morphism induced by the morphism $B'\to \mathbb B(R)$.
We know $m_{\mathbb B}\circ (i\otimes i)\circ (\phi'\otimes \phi')=i\circ \phi'\circ m_{\mathbb A}$. The morphism $m_{\mathbb B}\circ (i\otimes i)\circ (\phi'\otimes \phi')$ uniquely  factors  onto $\mathcal B'$, more concretely, through $m_{\mathcal B'}\circ (\phi'\otimes\phi')$.
The morphism $i\circ \phi'\circ m_{\mathbb A}$ uniquely  factors  onto $\mathcal B'$, effectively, through $\phi'\circ m_{\mathbb A}$. Then, $m_{\mathcal B'}\circ (\phi'\otimes\phi')=\phi'\circ m_{\mathbb A}$ and $\phi'$ is a morphism of $\mathcal R$-algebras.

\end{proof}

\begin{definition} \label{proquasi} We will say that a functor of $\mathcal R$-algebras  is a functor of proquasi-coherent algebras if it is the inverse limit of its quasi-coherent algebra quotients.\end{definition}


\begin{examples} \label{ejemplor} Quasi-coherent algebras are proquasi-coherent.

Let $R=K$ be a field, $A$ be a commutative $K$-algebra and $I\subseteq A$ be an ideal. Then, $\mathbb B=\plim{n\in\mathbb N}\mathcal A/\mathcal I^n\in\mathfrak F$ (by \ref{cansado}) and it is a proquasi-coherent algebra: $\mathbb B\simeq
\prod_n \mathcal I^n/\mathcal I^{n+1}$. Then, $\mathbb B^*=\oplus_n (\mathcal I^n/\mathcal I^{n+1})^*=\ilim{n} (\mathcal A/\mathcal I^n)^*$. Therefore, $\mathbb B^*$ is equal to the direct limit of the dual of the quasi-coherent algebra quotients of $\mathbb B$. Dually, $\mathbb B$ is a proquasi-coherent algebra.


\end{examples}

\begin{proposition} \label{5.9} Let $\mathcal C^*\in\mathfrak F$ be a functor of $\mathcal R$-algebras (i.e., a scheme of $\mathcal R$-algebras). Then, $\mathcal C^*$ is a functor of proquasi-coherent algebras.\end{proposition}

\begin{proof} 1. If $\mathcal M^*\in\mathfrak F$ and $f\colon \mathcal M^*\to \mathcal N$ is a morphism of functors of $\mathcal R$-modules, then $N':=\Ima f_R$ is a finitely generated $R$-module:  By Lemma \ref{invqua}, $f$ factors via an epimorphism
$f'\colon  \mathcal M^*\to \mathcal N'$. $\Hom_{\mathcal R}(\mathcal M^*,\mathcal N')=M\otimes N'$, then $f'=m_1\otimes n_1+\cdots+m_r\otimes n_r$, for some $m_i\in M$ and $n_i\in N'$. Hence,
$f(w)=\sum_i w(m_i)\cdot n_i$, for all $w\in\mathcal M^*$. Therefore, $N'=\langle n_1,\ldots,n_r\rangle$.

2. $\mathcal C^*$ is a left and right $\mathcal C^*$-module, then $\mathcal C$ is a right and left $\mathcal C^*$-module. Given $c\in C$, the dual morphism of the morphism
$\mathcal C^*\to \mathcal C$, $w\mapsto w\cdot c$ is the morphism $\mathcal C^*\to \mathcal C$, $w\mapsto c\cdot w$.

3. $C$ is the direct limit of its finitely generated $R$-submodules. Let $N=\langle n_1,\ldots, n_r\rangle\subset C$ be a finitely generated  $R$-module  and let $f\colon \mathcal C^{*r}\to \mathcal N$ be defined by $f((w_i)):=\sum_i w_i\cdot n_i$. Then,
$N':=\Ima f_R$ is a finitely generated $R$-module.
By Proposition \ref{invqua4}, $\mathcal N'$ is a quasi-coherent $\mathcal C^*$-submodule of $\mathcal C$.
Write $N'=\langle n_1,\ldots, n_s\rangle$. The morphism $\mathbb End_{\mathcal R}(\mathcal N')\to \oplus^s\mathcal N'$, $g\mapsto (g(n_i))_i$ is injective. By Proposition \ref{invqua2}, the morphism of functors of $\mathcal R$-algebras $\mathcal C^*\to \mathbb End_{\mathcal R}(\mathcal N')$ $w\mapsto w\cdot$  factors through an epimorphism onto a
quasi-coherent algebra, $\mathcal B'$. The dual morphism of the composite morphism
$$\xymatrix{\mathcal C^* \ar@{->>}[r] & \mathcal B'\ar[r] & \mathbb End_{\mathcal R}(\mathcal N') \ar@{^{(}->}[r] & \oplus^s \mathcal N' \ar[r] &  \oplus^s \mathcal C\ar[r]^-{\pi_i} \ar[r] &\mathcal C\\ w \ar[rrrrr] &&&&& w\cdot n_i}$$
is $\mathcal C^*\to {\mathcal B'}^* \hookrightarrow \mathcal C$, $w\mapsto n_i\cdot w$.
Hence, $n_i\in {B'}^*$, for all $i$, and $N'\subseteq B'^*$.  Therefore, $\mathcal C$ is equal to the direct limit of the dual
functors of the quasi-coherent algebra quotients of $\mathcal C^*$. Dually,
$\mathcal C^*$ is a functor of proquasi-coherent algebras.

\end{proof}

\begin{lemma} \label{harto} Let $\mathbb M_1,\ldots,\mathbb M_n\in\mathfrak F$ and let us consider the natural morphism $\mathbb M_1\otimes\cdots\otimes\mathbb M_n\to (\mathbb M_1^*\otimes\cdots\otimes\mathbb M_n^*)^*$. Then, the induced morphism

$$(\mathbb M_1\otimes\cdots\otimes\mathbb M_n)^{**}\to (\mathbb M_1^*\otimes\cdots\otimes\mathbb M_n^*)^*$$ is injective.\end{lemma}

\begin{proof} Let us follow the notations $\oplus_{J_i}\mathcal R\subseteq \mathbb M_i\subseteq \prod_{J_i}\mathcal R$.
By Lemma \ref{uf}, $\oplus_{J_i}\mathcal R\subseteq \mathbb M_i^*\subseteq \prod_{J_i}\mathcal R$. By induction  hypothesis, $\oplus_{J_1\times\cdots\times J_{n-1}}\mathcal R\subseteq (\mathbb M_1\otimes\cdots\otimes\mathbb M_{n-1})^{*}\subseteq \prod_{J_1\times\cdots\times J_{n-1}}\mathcal R$.
Since
$$(\mathbb M_1\otimes\cdots\otimes\mathbb M_n)^*=\mathbb Hom_{\mathcal R}(\mathbb M_1\otimes\cdots\otimes\mathbb M_{n-1},\mathbb M_n^*)=\mathbb Hom_{\mathcal R}((\mathbb M_1\otimes\cdots\otimes\mathbb M_{n-1})^{**},\mathbb M_n^*),$$
by Equation \ref{EQ2}, $\oplus_{J_1\times\cdots\times J_n}\mathcal R\subseteq (\mathbb M_1\otimes\cdots\otimes\mathbb M_n)^*\subseteq \prod_{J_1\times\cdots\times J_n}\mathcal R$.
Hence, firstly  $(\mathbb M_1\otimes\cdots\otimes\mathbb M_n)^{**}\subseteq \prod_{J_1\times\cdots\times J_n}\mathcal R$, by Lemma \ref{uf}, secondly
$(\mathbb M_1^*\otimes\cdots\otimes\mathbb M_n^*)^*\subseteq \prod_{J_1\times\cdots\times J_n}\mathcal R$.

As a consequence, the natural morphism
$(\mathbb M_1\otimes\cdots\otimes\mathbb M_n)^{**}\to (\mathbb M_1^*\otimes\cdots\otimes\mathbb M_n^*)^*$ is injective.

\end{proof}

\begin{theorem} \label{prodspec} Let $\mathbb A,\mathbb B\in\mathfrak F$ be two functors of proquasi-coherent algebras. Then,
$(\mathbb A^*\otimes \mathbb B^*)^*\in \mathfrak F$ is a functor of proquasi-coherent algebras and it holds
$$\Hom_{\mathcal R-alg}(\mathbb A\otimes \mathbb B,\mathbb C)=\Hom_{\mathcal R-alg}((\mathbb A^*\otimes \mathbb B^*)^*,\mathbb C)$$
for every functor of proquasi-coherent algebras $\mathbb C$.

\end{theorem}

\begin{proof} Write $\mathbb A=\plim{i}\mathcal A_i$ and $\mathbb B=\plim{j}\mathcal B_j$. Observe that
$$\aligned(\mathbb A^* \otimes \mathbb B^*)^* & =\mathbb Hom_{\mathcal R}(\mathbb A^*,\mathbb B)\overset{\text{\ref{3.2}}}=
\mathbb Hom_{\mathcal R}(\ilim{i} \mathcal A_i^*,\mathbb B)=
\mathbb Hom_{\mathcal R}(\ilim{i} \mathcal A_i^*,\plim{j} \mathcal B_j)\\ & =\plim{i,j}
\mathbb Hom_{\mathcal R}( \mathcal A_i^*,\mathcal B_j)=
\plim{i,j} (\mathcal A_i\otimes \mathcal B_j)\endaligned$$
Then, $(\mathbb A^*\otimes \mathbb B^*)^*$ is a functor of algebras and
the natural morphism $\mathbb A\otimes \mathbb B\to (\mathbb A^*\otimes \mathbb B^*)^*$
is a morphism of functors of algebras.

Given a morphism of functor of $\mathcal R$-algebras $\phi\colon \mathbb A\otimes\mathbb B\to \mathcal C$, let $\phi_1=\phi_{|\mathbb A\otimes 1}$ and $\phi_2=\phi_{|1\otimes\mathbb B}$. Then, $\phi_1$ factors through an epimorphism onto  a quasi-coherent algebra quotient $\mathcal A_i$ of $\mathbb A$, and  $\phi_2$ factors through an epimorphism onto a quasi-coherent algebra quotient $\mathcal B_j$ of $\mathbb B$. Then, $\phi$ factors  through $\mathcal A_i\otimes\mathcal B_j$, and $\phi$ factors through $(\mathbb A^*\otimes \mathbb B^*)^*$.
Then, $$\Hom_{\mathcal R-alg}((\mathbb A^*\otimes \mathbb B^*)^*,\mathcal C)\to\Hom_{\mathcal R-alg}(\mathbb A\otimes \mathbb B,\mathcal C)$$
is surjective. It is also injective, because
$$\aligned \Hom_{\mathcal R}((\mathbb A^*\otimes \mathbb B^*)^*,\mathcal C)& =\Hom_{\mathcal R}(\mathcal C^*,(\mathbb A^*\otimes \mathbb B^*)^{**})\\ & \overset{\text{\ref{harto}}}\subseteq \Hom_{\mathcal R}(\mathcal C^*,(\mathbb A\otimes \mathbb B)^*)=\Hom_{\mathcal R}(\mathbb A\otimes \mathbb B,\mathcal C)\endaligned$$

Then, $\Hom_{\mathcal R-alg}(\mathbb A\otimes \mathbb B,\mathbb C)=\Hom_{\mathcal R-alg}((\mathbb A^*\otimes \mathbb B^*)^*,\mathbb C)$
for every proquasi-coherent algebra $\mathbb C$.

A morphism of functors of algebras $f\colon (\mathbb A^*\otimes \mathbb B^*)^*\to\mathcal C$ factors through some $\mathcal A_i\otimes\mathcal B_j$ because
$f_{|\mathbb A\otimes\mathbb B}$ factors through some $\mathcal A_i\otimes\mathcal B_j$. Then, the inverse limit of the quasi-coherent algebra quotients of
$(\mathbb A^*\otimes \mathbb B^*)^*$ is equal to $\plim{i,j} (\mathcal A_i\otimes\mathcal B_j)= (\mathbb A^*\otimes \mathbb B^*)^*$, that is,
$(\mathbb A^*\otimes \mathbb B^*)^*$ is a proquasi-coherent algebra.

\end{proof}

\begin{notation} Let $\mathbb A_1,\ldots,\mathbb A_n\in\mathfrak F$ be functors of proquasi-coherent $\mathcal R$-algebras.
 We denote $\mathbb A_1\tilde\otimes\cdots\tilde\otimes\mathbb A_n:=(\mathbb A_1^*\otimes\cdots\otimes\mathbb A_n^*)^{*},$
which is the closure of functors of proquasi-coherent algebras of $\mathbb A_1\otimes\cdots\otimes\mathbb A_n$.\end{notation}
Observe that $\mathbb A_1\tilde\otimes(\mathbb A_2\tilde\otimes\cdots\tilde\otimes\mathbb A_n)=\mathbb A_1\tilde\otimes\cdots\tilde\otimes\mathbb A_n$.

\section{Functors of bialgebras}

\begin{definition} \label{bialgebras} A functor  $\mathbb B\in \mathfrak F$
is said to be a functor of proquasi-coherent bialgebras when $\mathbb B$ and
$\mathbb B^*$ are functors of  proquasi-coherent $\mathcal R$-algebras such that the dual morphisms of the multiplication morphism $m\colon \mathbb B^*\otimes \mathbb B^*\to \mathbb B^*$ and the unit morphism $u\colon \mathcal R\to \mathbb B^*$ are morphisms of functors of $\mathcal R$-algebras.

Let $\mathbb B,\mathbb B'$ be two functors of proquasi-coherent bialgebras. We will say that a morphism of $\mathcal R$-modules, $f\colon \mathbb B\to\mathbb B'$ is a morphism of functors of bialgebras if $f$ and $f^* \colon {\mathbb B'}^*\to\mathbb B^*$ are morphisms of functors of $\mathcal R$-algebras.

\end{definition}

\begin{note} Let $\mathbb B\in\mathfrak F$ be a functor of proquasi-coherent algebras.
Defining a multiplication morphism $\mathbb B^*\otimes \mathbb B^*\to \mathbb B^*$
(associative and with a unit) is equivalent to defining a comultiplication morphism
$\mathbb B\to \mathbb B\tilde\otimes\mathbb B$ (coassociative and with a counit), because
$$\Hom_{\mathcal R}(\mathbb B^*\otimes\overset n\cdots\otimes \mathbb B^*,\mathbb B^*)=
\Hom_{\mathcal R}(\mathbb B,(\mathbb B^*\otimes\overset n\cdots\otimes \mathbb B^*)^*)
=\Hom_{\mathcal R}(\mathbb B,\mathbb B\tilde\otimes\overset n\cdots\tilde\otimes \mathbb B)$$

\end{note}

In the literature, an $R$-algebra $A$ is said to be a bialgebra if it is a coalgebra (with counit) and the comultiplication  $c \colon A\to A\otimes_R A$ and the counit $e\colon A\to R$ are  morphisms of $R$-algebras. We will say that a bialgebra $A$ is a free $R$-bialgebra if $A$ is a free $R$-module. We will say that a functor  of proquasi-coherent bialgebras $\mathbb B\in\mathfrak F$ is a functor of quasi-coherent bialgebras if $\mathbb B$ is
a quasi-coherent $\mathcal R$-module. Observe that if
$\mathcal A^*\in\mathfrak F$ is a functor of $\mathcal R$-algebras then
it is a proquasi-coherent algebra, by Proposition \ref{5.9}.

\begin{proposition} \label{5.24} The functors $A \rightsquigarrow \mathcal A$ and $\mathcal A \rightsquigarrow \mathcal A(R)$ establish an equivalence between the category of free $R$-bialgebras and the category of functors of $\mathcal R$-quasi-coherent bialgebras.\end{proposition}

\begin{theorem} \label{dualbial} Let ${\mathcal C}_{\mathfrak F-Bialg.}$ be the category of
functors $\mathbb B\in\mathfrak F$ of proquasi-coherent bialgebras. The functor ${\mathcal C}_{\mathfrak F-Bialg.}\rightsquigarrow{\mathcal C}_{\mathfrak F-Bialg.}$, $\mathbb B \rightsquigarrow {\mathbb B}^*$ is a categorical anti-equivalence.
\end{theorem}

\begin{proof}  Let $\{\mathbb B,m,u; \mathbb B^*,m',u'\}$  be a functor of bialgebras. Let us only check that $m^*\colon \mathbb B^*\to (\mathbb B\otimes\mathbb B)^*=\mathbb B^*\tilde\otimes\mathbb B^*$ is a morphism of functors of algebras.
By hypothesis, ${m'}^*\colon \mathbb B\to (\mathbb B^*\otimes \mathbb B^*)^*=\mathbb B\tilde\otimes\mathbb B$ is a morphism of functors of algebras.
We have the commutative square:
$$\xymatrix{ \mathbb B \ar[rr]^-{{m'}^*} & &\mathbb B\tilde\otimes \mathbb B \\ \mathbb B \otimes \mathbb B \ar[u]^-m \ar[rr]^-{{m'}^*_{13}\otimes {m'}^*_{24}}& & \mathbb B\tilde \otimes \mathbb B\tilde \otimes \mathbb B\tilde \otimes \mathbb B \ar[u]_-{m\otimes m}} $$
where $({m'}^*_{13}\otimes {m'}^*_{24})(b_1\otimes b_2):=
\sigma({m'}^*(b_1)\otimes {m'}^*(b_2))$ and $\sigma(b_1\otimes b_2\otimes b_3\otimes b_4):=b_1\otimes b_3\otimes b_2\otimes b_4$, for all $b_i\in\mathbb B$.
Taking duals, we obtain the commutative diagram:

$$\xymatrix{ \mathbb B^* \ar[d]_-{m^*} &  &
(\mathbb B^*\otimes \mathbb B^*)^{**} \ar[ll] \ar[d]^-{(m\otimes m)^*} & \mathbb B^*\otimes \mathbb B^* \ar@{-->}[l] \ar@{-->}@/_{5mm}/[lll]_-{m'} \ar@{-->}@/^{12mm}/[ddll]^-{m^*\otimes m^*}
\\ \mathbb B^*\tilde\otimes \mathbb B^* &  &
(\mathbb B^*\otimes \mathbb B^*\otimes  \mathbb B^*\otimes \mathbb B^*)^{**}
\ar[ll] \ar@{-->}[dl] & \\& \mathbb B^*\tilde\otimes \mathbb B^*\tilde\otimes  \mathbb B^*\tilde\otimes \mathbb B^* \ar@{-->}[ul]_-{{m'}_{13}\otimes {m'}_{24}} & &}$$
which says that $m^*$ is a morphism of functors of $\mathcal R$-algebras.

\end{proof}

In \cite[Ch. I, \S 2, 13]{dieudonne}, Dieudonné proves the anti-equivalence between
the category of commutative $K$-bialgebras and the  category of
linearly compact  cocommutative $K$-bialgebras (where $K$ is a field).

Let $C$ be a $\mathcal R$-module. $C$ is a $R$-coalgebra (coassociative with counit) if and only if $\mathcal C^*$ is a functor of algebras: Observe that
$$\aligned (\mathcal C^*\otimes\overset n\cdots\otimes \mathcal C^*)^* & =
\mathbb Hom_{\mathcal R}(\mathcal C^*\otimes\overset n\cdots\otimes \mathcal C^*,\mathcal R)=\mathbb Hom_{\mathcal R}(\mathcal C^*\otimes\overset{n-1}\cdots\otimes \mathcal C^*,\mathcal C)\\ & \underset{\text{\ref{prop4}}}=\mathbb Hom_{\mathcal R}(\mathcal C^*\otimes\overset{n-2}\cdots\otimes \mathcal C^*,\mathcal C\otimes \mathcal C)=\cdots=\mathcal C\otimes\overset n\cdots\otimes \mathcal C\endaligned$$

\begin{definition} The functor of algebras $\mathcal C^*$ is called algebra scheme. If $\mathcal C^*\in\mathfrak F$ is a functor of bialgebras then 
it is called bialgebra scheme.\end{definition}

\begin{notation} Let $\mathbb A$ be a reflexive functor of $\mathcal K$-algebras and
let $\{\mathcal A_i\}$ be the set of quasi-coherent quotients of $\mathbb A$ such that
$\dim_K A_i<\infty$. We denote $\bar{\mathbb A}:=\plim{i} \mathcal A_i$ which
is an algebra scheme because $\mathcal A_i^*$ is quasi-coherent and $\plim{i} \mathcal A_i=(\ilim{i} \mathcal A_i^*)^*$.\end{notation}

\begin{note} Given a functor of algebras $\mathcal C^*\in\mathfrak F$, then
$\bar{\mathcal C^*}=\mathcal C^*$ (see proof of Proposition \ref{5.9}).\end{note}

\begin{proposition} \cite[5.9]{Amel} \label{a5.9} Let $\mathbb A$ be a reflexive functor of $\mathcal K$-algebras. Then,
$$\Hom_{\mathcal K-alg}(\mathbb A,\mathcal C^*)=
\Hom_{\mathcal K-alg}(\bar{\mathbb A},\mathcal C^*)$$
for all algebra schemes $\mathcal C^*$.
\end{proposition}

\begin{theorem} Let $\mathbb B\in\mathfrak F$ be a functor of proquasi-coherent $\mathcal K$-bialgebras. Then, $\bar{\mathbb B}$ is a scheme of bialgebras and

$$\Hom_{\mathcal K-bialg}(\mathbb B,\mathcal C^*)=
\Hom_{\mathcal K-bialg}(\bar{\mathbb B},\mathcal C^*)$$
for all bialgebra schemes $\mathcal C^*$.
\end{theorem}

\begin{proof} Given any $\mathbb A_1,\ldots,\mathbb A_n\in\mathfrak F$ proquasi-coherent algebras then $\overline{\mathbb A_1\tilde\otimes\ldots\tilde\otimes \mathbb A_n}=
\bar{\mathbb A_1}\tilde\otimes\ldots\tilde\otimes \bar{\mathbb A_n}$, by Proposition \ref{a5.9}.
Then, the comultiplication morphism $\mathbb B\to \mathbb B\tilde\otimes \mathbb B$ defines a comultiplication morphism $\bar{\mathbb B}\to \bar{\mathbb B}\tilde\otimes \bar{\mathbb B}$, and
$\bar{\mathbb B}$ is a scheme of bialgebras.

Given a morphism of functors of bialgebras $f\colon \mathbb B\to \mathcal C^*$, that is, a morphism of functors of algebras such that the diagram
$$\xymatrix{\mathbb B \ar[r] \ar[d]^-f & \mathbb B\tilde\otimes\mathbb B\ar[d]^-{f\otimes f}\\ \mathcal C^* \ar[r] & \mathcal C^*\tilde\otimes\mathcal C^*}$$
is commutative, the induced morphism of functors algebras $\bar{\mathbb B}\to \mathcal C^*$ is a morphism of functors of bialgebras. Reciprocally, given a morphism of bialgebras $\bar{\mathbb B}\to \mathcal C^*$, the composition morphism
$\mathbb B\to \bar{\mathbb B}\to \mathcal C^*$ is a mosphism of functors of bialgebras.

\end{proof}

\begin{corollary} \label{cor1.7} Let $A$ and $B$ be  $K$-bialgebras. Then,
$$\Hom_{\mathcal K-bialg}(\bar{\mathcal A},{\mathcal B^*})=\Hom_{\mathcal K-bialg}(\bar {\mathcal B},{\mathcal A^*})$$
\end{corollary}

\begin{proof} It holds that
$\Hom_{\mathcal K-bialg}(\bar{\mathcal A},{\mathcal B^*})=\Hom_{\mathcal K-bialg}({\mathcal A},{\mathcal B^*})=\Hom_{\mathcal K-bialg}({\mathcal B},{\mathcal A^*})$ $=\Hom_{\mathcal K-bialg}(\bar {\mathcal B},{\mathcal A^*})$.

\end{proof}

\begin{note} \label{cite{E}} The bialgebra $A^\circ :=\Hom_{\mathcal K}(\bar{\mathcal A},\mathcal K)$ is sometimes known as the ``dual bialgebra" of $A$ and Corollary \ref{cor1.7}  says (dually) that the functor assigning to each bialgebra its dual bialgebra is autoadjoint (see \cite[3.5]{E}).
\end{note}

\end{document}